\documentclass{amsart}[11pt]
\usepackage{amssymb,mathrsfs,amscd,graphicx, url,bbm}

\newcommand{\abs}[1]{\left\lvert #1 \right\rvert}
\newcommand{\zmod}[1]{\mathbb{Z}/ #1 \mathbb{Z}}

\newcommand{\dangle}[1]{\left\langle #1 \right\rangle}
\newcommand{\bmu}{\mathbf{\mu}}

\def\nw{\mathrm{new}}
\DeclareMathSymbol{\twoheadrightarrow} {\mathrel}{AMSa}{"10}

\DeclareMathOperator{\Jac}{Jac}
\DeclareMathOperator{\Char}{{Char\,}}

\DeclareMathOperator{\rk}{rank}
\DeclareMathOperator{\supp}{supp}

\DeclareMathOperator{\Div}{Div}
\DeclareMathOperator{\Disc}{Disc}
\DeclareMathOperator{\Pic}{Pic}

\DeclareMathOperator{\Res}{Res}

\DeclareMathOperator{\Aut}{Aut}
\DeclareMathOperator{\End}{End}
\DeclareMathOperator{\Hom}{Hom}

\DeclareMathOperator{\Id}{Id}
\DeclareMathOperator{\Gal}{Gal}
\DeclareMathOperator{\Mat}{Mat}

\DeclareMathOperator{\Nm}{N}  %%% norm
\DeclareMathOperator{\Lie}{Lie}

\def\a{{\mathfrak a}} %%%command \a was predefined 
\newcommand{\m}{\mathfrak{m}}

\newcommand{\FV}{\mathfrak{F}} 
\newcommand{\GV}{\mathfrak{G}}

\newcommand{\RV}{\mathfrak{R}}

\newcommand{\LZ}{\mathscr{L}}

\newcommand{\cc}{\mathbb{C}}
\newcommand{\ff}{\mathbb{F}}
\newcommand{\nn}{\mathbb{N}}

\newcommand{\pp}{\mathbb{P}}
\newcommand{\qq}{\mathbb{Q}}

\newcommand{\zz}{\mathbb{Z}}
\def\bmu{\boldsymbol \mu}
\newcommand{\DD}{\mathcal{D}}
\newcommand{\EE}{\mathcal{E}}

\newcommand{\MM}{\mathcal{M}}

\newcommand{\OO}{\mathcal{O}}
\newcommand{\PP}{\mathcal{P}}
\newcommand{\QQ}{\mathcal{Q}}

\def\Sn{{\mathbf S}_n}
\def\An{{\mathbf A}_n}
\newtheorem{thm}{Theorem}[section]
\newtheorem{lem}[thm]{Lemma}
\newtheorem{cor}[thm]{Corollary}
\newtheorem{prop}[thm]{Proposition}
\theoremstyle{definition}

\newtheorem{rem}[thm]{Remark}

\newtheorem{sect}[thm]{}

\numberwithin{equation}{section}

\newtheoremstyle{notitle}  % this product a paragraph that's
  {}% measure of space to leave above the theorem. E.g.: 3pt
  {}% measure of space to leave below the theorem. E.g.: 3pt
  {\itshape}% name of font to use in the body of the theorem
  {}% measure of space to indent
  {}% name of head font
  {\ }% punctuation between head and body
  {.5em}% space after theorem head; " " = normal interword space
  {}% Man
\theoremstyle{notitle}

\title[Fixed points and
  homology of superelliptic Jacobians]{Fixed points and
  homology of superelliptic Jacobians}

\author{Haining Wang, Jiangwei Xue, Chia-Fu Yu}

\address{(Wang) Department of Mathematics, Pennsylvania State
  University, University Part, PA 16802, USA} 

\email{wang\_h\char`\@math.psu.edu}

\address{(Xue) Institute of Mathematics, Academia Sinica, 6F,
  Astronomy-Mathematics Building, No. 1, Sec. 4, Roosevelt Road,
  Taipei 10617, TAIWAN.} 

\email{xue\_j\char`\@math.sinica.edu.tw}

\address{
(Yu) Institute of Mathematics, Academia Sinica and NCTS (Taipei Office)\\
Astronomy-Mathematics Building \\
No. 1, Roosevelt Rd. Sec. 4 \\
Taipei, 10617, Taiwan}
\email{chiafu@math.sinica.edu.tw}

\begin{document}
\date{\today} 

\begin{abstract}
  Let $\eta: C_{f,N}\to \pp^1$ be a cyclic cover of $\pp^1$ of degree
  $N$ which is totally and tamely ramified for all the ramification
  points.  We determine the group of fixed points of the cyclic group
  $\bmu_N\cong \zmod{N}$ acting on the Jacobian
  $J_N:=\Jac(C_{f,N})$. For each prime $\ell$ distinct from the
  characteristic of the base field, the Tate module $T_\ell J_N$ is
  shown to be a free module over the ring
  $\zz_\ell[T]/(\sum_{i=0}^{N-1}T^i)$. We also calculate the degree of
  the induced polarization on the new part $J_N^\nw$ of the Jacobian.
\end{abstract}

\maketitle
%%%%%%%%%%%%%%%%%%%%%%%start here%%%%%%%%%%%%%%%%%%%%%%%%%%%%%%%%%%%%%%%%
%%%%%%%%%%%%%%%%%%%How to add to content line %%%%%%%%%%%%%%%%%%%%%
%\addcontentsline{toc}{chapter}{\protect\numberline{}Appendix}
\section{Introduction}\label{sec:introduction}
Through out this paper, $K$ is an algebraically closed field except
when specified otherwise. The characteristic of $K$ is denoted by
$\Char(K)$. If $A$ is an abelian variety over $K$, we write $A^\vee$
for the dual abelian variety of $A$, and $\End(A)$ for the
endomorphism ring of $A$. The endomorphism algebra
$\End^0(A):=\End(A)\otimes_\zz \qq$ is a finite-dimensional semisimple
algebra over $\qq$.  Given an abelian group $G$ (or a commutative
group scheme $G$ over $K$), $G[m]$ denotes the kernel of the
homomorphism $G\xrightarrow{m}G$. We often identify a finite \'etale
group scheme $G/K$ with $G(K)$. The cardinality of a finite set $S$ is
denoted by $\abs{S}$. In particular, for any prime $\ell\neq
\Char(K)$, one has $\abs{A[\ell]}=\ell^{2\dim A}$.  The letters $p$
and $\ell$ always denote primes in $\nn$.

Fix an integer $N>1$ coprime to $\Char(K)$. Then $T^N-1\in K[T]$ is
separable over $K$.  Let $\xi_N\in K$ be a primitive $N$-th root of
unity in $K$, and $\bmu_N:=\dangle{\xi_N}$, the group of $N$-th root
of unity in $K$. Suppose that $f(X)=\prod_{i=1}^n
(X-\alpha_i)^{e_i}\in K[X]$ is a monic polynomial with
\begin{equation}
  \label{eq:assumption-on-f}
\gcd(\deg(f), N)=1, \qquad \gcd(e_i,N)=1, \; \forall\, 1\leq i\leq n.
\end{equation}
For example, if $N$ is even, then all $e_i$ must be odd, and hence
$n$ must be odd as well to ensure that $\gcd(\deg(f), N)=1$.

Let $C_{f,N}$ be the smooth projective curve defined by the affine
equation $Y^N=f(X)$, and $J_N:=\Jac(C_{f,N})$ be the Jacobian variety
of $C_{f,N}$.  The map \[\eta: C_{f,N}\to \pp^1, \qquad (X,Y)\mapsto
X\] realizes $C_{f,N}$ as a cyclic cover of degree $N$ of the
projective line $\pp^1$. There is a canonical isomorphism $\rho_C:
\bmu_N\xrightarrow{\cong} \Aut(\eta)\subseteq \Aut_K(C_{f,N})$, given
by
\[ \rho_C(\xi): C_{f,N}\to C_{f,N}, \qquad (X, Y)\mapsto (X, \xi Y),
\quad \forall\, \xi\in \bmu_N.\] We denote $\rho_C(\xi_N)$ by
$\delta_N$.

The Albanese functoriality induces an action $\rho_J: \bmu_N\to
\Aut(J_N)$ of $\bmu_N$ on the Jacobian $J_N$. By an abuse of notation,
we still write $\delta_N$ for the map $\rho_J(\xi_N): J_N\to J_N$
induced from $\delta_N: C_{f,N}\to C_{f,N}$. For each $D\in \nn$, let
$\Phi_D(T)\in \zz[T]$ be the $D$-th cyclotomic polynomial, which is a
monic irreducible polynomial of degree $\varphi(D)$. It will be shown
in Subsection \ref{subsec:characteristic-poly-of-deltaN} (cf. also
\cite[Lemma 4.8]{MR2166091} in the case $N=p^r$ is a prime power and
$\Char(K)=0$) that the minimal polynomial over $\zz$ of
$\delta_N\in \End(J_N)$ is
\begin{equation}
  \label{eq:def-of-PN}
  P_N(T):=\frac{T^N-1}{T-1} = \prod_{D\mid N, D> 1}
  \Phi_D(T)=\sum_{i=0}^{N-1}T^i.  
\end{equation}
So there is an embedding 
\begin{equation}
  \label{eq:embedding-into-end-of-JN}
  \iota
  :\zz[T]/(P_N(T))\hookrightarrow \End(J_N), \qquad T\mapsto \delta_N .
\end{equation}
Hence for each prime $\ell\neq \Char(K)$, the Tate-module $T_\ell
J_N:=\varprojlim_{i\geq 1} J_N[\ell^i]$ is naturally a
$\zz_\ell[T]/(P_N(T))$-module.
\begin{thm}[Main Theorem]\label{thm:main}
  For any prime $\ell\neq \Char(K)$, $T_\ell J_N$ is a free
  $\zz_\ell[T]/(P_N(T))$-module of rank $n-1$.  In particular, if
  $K=\cc$, the first homology group $H_1(C_{f,N}(\cc), \zz)$ is a
  projective $\zz[T]/(P_N(T))$-module of rank $n-1$.
\end{thm}

Let $\zeta_N\in \bar{\qq}$ be a primitive $N$-th root of unity in a
fixed algebraic closure $\bar{\qq}$ of $\qq$. For each $D\mid N$, we
set $\zeta_D:=\zeta_N^{N/D}$. The embedding $\iota$ induces an
embedding
\begin{equation}
  \label{eq:13}
  \iota: \qq[T]/(P_N(T))\cong \prod_{D\mid N, D> 1}
  \qq(\zeta_D)\hookrightarrow \End^0(J_N).
\end{equation}

\begin{cor}\label{cor:optimal-embedding}
  We have $\End(J_N)\cap (\qq[T]/(P_N(T))) = \zz[T]/(P_N(T))$, where
  the intersection is taken within $\End^0(J_N)$.  In other words, the
  embedding $\iota: \zz[T]/(P_N(T))\hookrightarrow \End(J_N)$ is
  optimal.
\end{cor}

\begin{thm}\label{thm:results-on-JN-new}
  Given any $G(T)=\prod_{s=1}^t\Phi_{D_s}(T)\in \zz[T]$ such that
  $G(T)\mid P_N(T)$, the kernel of $G(\delta_N): J_N\to J_N$ is an
  abelian subvariety of $J_N$ of dimension $(n-1)\deg
  G(T)/2$. Moreover, $\ker G(\delta_N)=\sum_{s=1}^t \ker
  \Phi_{D_s}(\delta_N)$.  Let $J_N^\nw=\ker \Phi_N(\delta_N)$, then
  $J_N^\nw$ is isomorphic to the dual $(J_N^\nw)^\vee$. If $q=p^r$ is
  a prime power with $p\neq 2$, then there exists a principal
  polarization $\widetilde{\lambda}_q^\nw: J_q^\nw \to
  (J_q^\nw)^\vee$.
\end{thm} 

\begin{rem}
  In this remark, the ground field $K$ is not necessarily assumed to
  be algebraically closed. Let $\bar{K}$ be its algebraic
  closure. Since $J_N^\nw=\ker \Phi_N(\delta_N)$, there exists an
  embedding $\zz[\zeta_N]\hookrightarrow \End(J_N^\nw)$ given by
  $\zeta_N\mapsto \delta_N\mid_{J_N^\nw}$. Suppose that $\Char(K)=0$,
  $q:=N=p^r$ is a prime power, and $f(x)$ has no multiple roots.  In a
  series of papers \cite{MR2040573}, \cite{MR2166091},
  \cite{MR2349666}, \cite{MR2471095}, Yuri G. Zarhin showed that
  $\End_{\bar{K}}(J_q^\nw)=\zz[\zeta_q]$ assuming that $\deg f(x)\geq
  4$ and $f(x)$ is irreducible over $K$ with ``large'' Galois group
  (For example, $\Gal(f)$ is either the full symmetric group $\Sn$ or
  the alternating group $\An$ when $\deg f(x)\geq 5$, or
  $\Gal(f)=\mathrm{S}_4$ when $\deg f(x)=4$). When $K=\cc$ and $\deg
  f(x)=3$, the endomorphism algebra $\End^0(J_q^\nw)$ has been
  classified. In particular, if $p>7$, then $\End^0(J_q^\nw)$ is
  either $\qq(\zeta_q)$, a quadratic field extension of
  $\qq(\zeta_q)$, or $\qq(\zeta_q)\oplus \qq(\zeta_q)$. The generic
  case was treated in \cite{MR2349666} by Zarhin and the
  classification was given in \cite{xue-yu} jointly by the second and
  third named authors. Now suppose $K\subseteq \cc$, $\deg f(x)\geq
  3$, and $\End^0_{\bar{K}}(J_q^\nw)=\qq(\zeta_q)$. With some further
  mild conditions on $q$, the special Mumford-Tate group of $J_q^\nw$
  has been determined in another series of papers \cite{MR2831828},
  \cite{MR2657449}, \cite{MR2644383} jointly by Zarhin and the second
  named author.
\end{rem}

The paper is organized as follows. Section 2 studies the kernel of
endomorphisms of abelian varieties. The theorems above and their
corollaries are proved in Section 3, where we study the superelliptic
Jacobian $J_N$ and apply the results obtained in Section 2. Section 4
contains some arithmetic results that are used in the previous
sections.

\textbf{Acknowledgment:} This paper originated as a question on the
group of invariant points of superelliptic Jacobians raised to the
second named author by Chao Li of Harvard University during the
Special Week in Arithmetic and Number Theory Conference (July 8--12,
2013) in the National Center for Theoretic Science (NCTS) in Hsinchu,
Taiwan. We would like to thank him for the fruitful discussions. The
first named author is currently visiting NCTS during the fall semester
of 2013. He is partially supported by National Science Foundation
grant DMS1101368.  The second named author is partially supported by
the grant NSC 102-2811-M-001-090. The third named author is partially
supported by the grants NSC 100-2628-M-001-006-MY4 and AS-98-CDA-M01.

\section{Decomposition of Abelian varieties}
Throughout this section, $F(T)$ and $G(T)$ denote polynomials in
$\zz[T]$. Let $A$ be an abelian variety of positive dimension over
$K$. The minimal polynomial of an endomorphism of $A$ is monic over
$\zz$ \cite[Theorem 19.4]{MR2514037}. If $\phi\in \End(A)$ has
minimal polynomial $P(T)\in \zz[T]$, then there is an embedding
$\zz[T]/(P(T))\hookrightarrow \End(A)$ with $T\mapsto \phi$.  Given
$F(T)\mid P(T)$, the kernel of $\beta:=F(\phi)$ is a group scheme over
$K$. In this section, we give a criterion to determine when $\ker
\beta$ is an abelian subvariety of $A$. This question turns out to be
closely related to the torsion subgroup of $A(K)$.

Suppose $R$ is a commutative ring. For any $F(T)\in \zz[T]$, we write
$\bar{F}(T):=F(T)\otimes 1\in \zz[T]\otimes_\zz R=R[T]$. For example,
if $R=\zmod{m}$, then $\bar{F}(T)$ is just $F(T)$ modulo $m$. For
each $m\in \zz$ coprime to $\Char(K)$, the $m$-torsion group
$A[m]\subset A(K)$ is naturally a $(\zmod{m})[T]/(\bar{P}(T))$-module.

\begin{sect}\label{subsec:resultant-and-generator-of-intersection}
  Let $F(T), G(T)\in \zz[T]$ be two monic polynomials with $\gcd(F(T),
  G(T))=1$.  The quotient ring $R:=\zz[T]/(F(T))$ is a free
  $\zz$-module of rank $\deg F(T)$. By an abuse of notation, we still
  write $G(T)$ for the canonical image of $G(T)$ in $R$.  The
  resultant of $F(T)$ and $G(T)$ is defined to be (See \cite[Section
  IV.6.6]{MR1080964})
  \begin{equation}
    \label{eq:19}
   \Res(F(T), G(T))= \Nm_{R/\zz} (G(T))=\prod_{F(x)=0, G(y)=0}
  (x-y)
  \end{equation}
  where $\Nm_{R/\zz}: R\to \zz$ is the norm map, and $x,y\in
  \bar{\qq}$ are roots of $F(T)$ and $G(T)$ respectively.  Since
  $F(T)$ and $G(T)$ are coprime, $\Res(F(T), G(T))\neq 0$. There
  exists $a(T), b(T)\in \zz[T]$ such that \begin{equation}
  \label{eq:gcd-reduced-lem}
  a(T)F(T)+b(T)G(T)=\Res(F(T), G(T))\in \zz.
\end{equation}
Both $a(T)$ and $b(T)$ are uniquely determined if we further require
that $\deg a(T)< \deg G(T)$ (or equivalently, $\deg b(T)< \deg
F(T)$). The resultant may be calculated as the determinant of a matrix
whose entries are coefficients of $F(T)$ and $G(T)$.

  Let $\bar{F}(T):=F(T)\otimes 1\in K[T]=\zz[T]\otimes_\zz K$, and define $\bar{G}[T]$
  similarly. Then $\bar{F}(T)$ and $\bar{G}(T)$ share a common root in
  $K$ if and only if $\Res(F(T), G(T))$ is divisible by $\Char(K)$.
  We write $\Disc F(T)\in \zz$ for the discriminant of $F(T)$
  (\cite[Section IV.6.7]{MR1080964}). Then $\bar{F}(T)$ is separable
  if and only if $\Disc(F(T))$ is coprime to $\Char(K)$. Clearly,
\[\Res(F(T), G(T))\mid \Disc(F(T)G(T)).\]
We refer to Subsection~\ref{subsec:arithmetic-resultant} for some
further discussion of $\Res(F(T), G(T))$. 
\end{sect}

\begin{lem}\label{lem:reduced-kernel}
  Let $P(T)=F(T)G(T)\in \zz[T]$ be the minimal polynomial of
  $\phi\in \End(A)$, and $\beta:=F(\phi),
  \gamma:=G(\phi)\in \End(A)$. Suppose that $\Res(F(T), G(T))$ is
  coprime to $\Char(K)$. Then both $\ker\beta$ and $\ker\gamma$ are
  reduced group schemes over $K$, and
\[ \dim \ker \beta+\dim \ker \gamma= \dim A.\]
\end{lem}

\begin{proof}
  Let $\Lie(\beta): \Lie(A)\to \Lie(A)$ be the induced endomorphism of
  the Lie algebra of $A$. To show that $\ker\beta$ is reduced, it is
  enough to prove that
  \[\dim \ker(\beta) =\dim_K(\ker (\Lie(\beta)))=\dim_K(\Lie
  (\ker\beta)).\] A priori, $\dim_K(\Lie(\ker(\beta)))\geq \dim
  \ker(\beta)$. Similarly for $\gamma$.

  The subring $\zz[\phi]\subseteq \End(A)$ generated by $\phi$ is
  isomorphic to $\zz[T]/(P(T))$. So $\Lie(A)$ carries a natural
  $\zz[T]/(P(T))\otimes_\zz K$-module structure. Since $\Res(F(T),
  G(T))$ is coprime to $\Char K$, $\bar{F}(T)$ and $\bar{G}(T)$
  share no common factors. By the Chinese Reminder Theorem,
  \[ \left(\zz[T]/(P(T))\right)\otimes_\zz K = K[T]/(\bar{P}(T))\cong
  K[T]/(\bar{F}(T))\oplus K[T]/(\bar{G}(T)).\] Correspondingly,
  $\Lie(A)=\Lie(A)_F\oplus \Lie(A)_G$. Here $\Lie(A)_F=\ker
  (\Lie(\beta))$, which is naturally equipped with a
  $K[T]/(\bar{F}(T))$-module structure, and $\Lie(A)_G=\ker
  (\Lie(\gamma))$, which has a natural $K[T]/(\bar{G}(T))$-module
  structure.

Necessarily, $F(T)$ and $G(T)$ are coprime over $\qq$. For simplicity,
let $m:=\Res(F(T), G(T))$. We may choose $a(T),b(T)\in \zz[T]$ and such
that (\ref{eq:gcd-reduced-lem}) holds.  Then $\ker\beta \cap \ker
\gamma \subseteq A[m]$, a finite \'etale group scheme over $K$.  Since
$A(K)$ is divisible (\cite[Application 6.2]{MR2514037}),
\[ A=\beta(A) +\gamma(A)\subseteq \ker \gamma + \ker \beta.\]
We have 
\[
\begin{split}
\dim A&= \dim \ker \beta + \dim \ker \gamma \leq \dim_K \ker
(\Lie(\beta))+\dim_K \ker (\Lie(\gamma))\\&=\dim_K \Lie(A)=\dim A.  
\end{split}
\]
It follows that $\dim \ker(\beta)=\dim_K\ker(\Lie(\beta))$ and
similarly for $\gamma$. 
\end{proof}

\begin{cor}\label{cor:criterion-for-kernel-connected}
  We keep the notation and assumptions of
  Lemma~\ref{lem:reduced-kernel}.  Let $d:=\dim \ker \beta$. For any
  prime $p\neq \Char(K)$, $\abs{(\ker \beta)[p]}\geq p^{2d}$ , and
  $\ker\beta $ is connected if and only if the equality holds for any
  $p\mid \Res(F(T), G(T))$.
\end{cor}
\begin{proof}
  Since $\ker \beta$ is reduced, its identity component $(\ker
  \beta)^\circ$ is an abelian subvariety of $A$, and $\ker \beta$ is
  an extension of $(\ker\beta )^\circ$ by a finite \'etale group
  scheme $\pi_0(\ker\beta)$ over $K$:
  \begin{equation}
    \label{eq:extension-ab-by-finite-etale-gpsch}
    0 \to (\ker\beta )^\circ \to \ker \beta \to \pi_0(\ker\beta)\to
    0.  
  \end{equation}
  Because $(\ker\beta)^\circ(K)$ is divisible, it follows from the
  Snake Lemma \cite[Exercise A.3.10]{MR1322960} that there is an
  exact sequence
  \begin{equation}
    \label{eq:extension-of-torsion-gp}
    0\to (\ker\beta)^\circ[p]\to (\ker \beta)[p]\to
    \pi_0(\ker\beta)[p]\to 0
  \end{equation}
  for any prime $p$. In particular, if $p\neq \Char(K)$,
  \begin{equation}
    \label{eq:11}
\abs{(\ker
    \beta)[p]}= \abs{(\ker\beta)^\circ[p]}\cdot
  \abs{\pi_0(\ker\beta)[p]}\geq \abs{(\ker\beta)^\circ[p]}=p^{2d}. 
  \end{equation}

  Recall that $\gamma(A)\subseteq (\ker\beta)^\circ$ and $\dim
  \gamma(A)= \dim A - \dim \ker\gamma = \dim (\ker\beta)^\circ$, so
  $(\ker\beta)^\circ = \gamma(A)$. Let $m:=\Res(F(T), G(T))$, and
  $a(T), b(T)\in \zz[T]$ be polynomials such that
  (\ref{eq:gcd-reduced-lem}) holds.  For all $x\in \ker\beta$, we
  have \[mx= a(\phi)\beta x + b(\phi) \gamma x= b(\phi)\gamma x\in
  \gamma(A)=(\ker\beta)^\circ.\] It follows that $\forall \,y\in
  \pi_0(\ker\beta)$, $my=0$.  Therefore, $\pi_0(\ker\beta)$ is trivial
  if and only if $\pi_0(\ker\beta)[p]$ is trivial for all $p\mid m$.
  By (\ref{eq:11}), this holds if and only if $\abs{(\ker \beta)[p]}=
  p^{2d}= \abs{(\ker\beta)^\circ[p]}$ for all $p\mid m$.
\end{proof}

\begin{lem}\label{lem:free-tate-module-p-torsion}
  Suppose that $\dim A =r \deg P(T)/2\in \nn$ for some $r\in \nn$.  Let
  $\ell$ be a prime distinct from $\Char(K)$.  The following are
  equivalent:
  \begin{enumerate}
  \item[(\ref{lem:free-tate-module-p-torsion}.i\phantom{i})] $A[\ell]$
    is a free $\ff_\ell[T]/(\bar{P}(T))$-module of rank $r$.
  \item[(\ref{lem:free-tate-module-p-torsion}.ii)] $T_\ell A$ is a
    free $\zz_\ell[T]/(P(T))$-module of rank $r$.
  \end{enumerate}
\end{lem}
\begin{proof}
  Since $A[\ell]\cong T_\ell A\otimes_{\zz_\ell} \ff_\ell$, clearly
  (\ref{lem:free-tate-module-p-torsion}.ii)$\Rightarrow$(\ref{lem:free-tate-module-p-torsion}.i). To
  show that
  (\ref{lem:free-tate-module-p-torsion}.i)$\Rightarrow$(\ref{lem:free-tate-module-p-torsion}.ii),
  it is enough to show that $A[\ell^i]$ is a free
  $(\zmod{\ell^i})[T]/(\bar{P}(T))$-module of rank $r$ for all $i\in
  \nn$.  Now fix $i$, and let $M:=A[\ell^i]$ and
  $R:=(\zmod{\ell^i})[T]/(\bar{P}(T))$. The ideal $\a=(\ell)\subseteq
  R$ is nilpotent. We have $M/\a M = M/\ell M\cong A[\ell]$. By
  assumption, $M/\a M$ is a free ($R/\a$)-module of rank $r$. It
  follows from Nakayama's lemma \cite[Corollary 4.8]{MR1322960} that
  $M$ can be generated by $r$ elements.  In other words, we have a
  surjective map $R^r\twoheadrightarrow M$. On the other hand,
  \[ \abs{M}= \abs{A[\ell^i]}= (\ell^i)^{2\dim A}=\ell^{ir\deg P(T)}=
  \abs{R^r}. \] Therefore, the map must be injective as well, and
  hence $M$ is free of rank $r$.
\end{proof}

We refer to \cite[Chapter 21]{MR1322960} for the concept of Gorenstein
rings.
\begin{lem}
  The Artinian ring $R=\zz[T]/(m, P(T))$ is Gorenstein for all
  positive integer $m>1$.
\end{lem}
\begin{proof}
  This follows directly from \cite[Corollary 21.19]{MR1322960} since
  $\zz[T]$ is a regular ring and $m, P(T)$ form a regular sequence.
\end{proof}

\begin{lem}\label{lem:gorenstein-ring-free}
  Let $(R,\m)$ be a local Artinian ring with residue field $k=R/\m$,
  and $M$ a finitely generated $R$-module of length $l_R(M)$. The
  socle of $M$ is defined to be the submodule $M_0:=\{x\in M \mid
  yx=0, \,\forall\, y\in \m\}$, which is the sum of all simple
  submodules of $M$.  Suppose $R$ is Gorenstein, then $l_R(M)\leq
  l_R(R)\dim_k(M_0)$, and the equality holds if and only if $M$ is a
  free $R$-module of rank $\dim_k M_0$.
\end{lem}
\begin{proof}
  For simplicity, we will write $l(M):=l_R(M)$ for the length of $M$
  if the ring $R$ is clear from context.  The socle of a local
  Gorenstein ring is simple (See \cite[Proposition 21.5]{MR1322960}),
  i.e., it has dimension $1$ over the residue field.  If $M$ is free,
  then $\dim_k M_0=\rk_R M=l(M)/l(R)$.

  Let $\MM$ be the category of all finitely generated $R$-modules.
  Since $R$ is Gorenstein, the functor $M \mapsto M^\vee:=\Hom_R(M,R),
  \forall \, M\in \MM$
  is a \textit{dualizing functor} from $\MM$ to itself. In other
  words,
  \begin{itemize}
  \item it is contravariant, $R$-linear and exact;
  \item  $\forall\, M\in \MM$, $(M^\vee)^\vee$ is canonically isomorphic
    to $M$.
  \end{itemize}
  In particular, the exactness implies that $l(M^\vee)=l(M)$, $\forall
  M\in \MM$.  We also note that $l(M_0)=\dim_k(M_0)$. For simplicity,
  let $r:=\dim_k M_0$.

  By definition, $M_0$ is the maximal submodule of $M$ annihilated by
  $\m$. Dualizing, we see that $M_0^\vee$ is the maximal quotient of
  $M^\vee$ annihilated by $\m$. That is, $M_0^\vee \cong M^\vee/(\m
  M^\vee)$. Therefore, \[\dim_kM^\vee/(\m M^\vee)=\dim_k
  M_0^\vee=l(M_0^\vee)=l(M_0)=\dim_k(M_0)=r.\]  By
  Nakayama's lemma, $M^\vee$ can be generated by $r$ elements. In
  other words, we have an exact sequence
  \[ 0 \to \ker \theta \to R^r\xrightarrow{\theta} M^\vee \to 0.\] 
Therefore,
\begin{equation}
  \label{eq:4}
l(M)=l(M^\vee)= 
  l(R)r -l(\ker\theta)\leq l(R)r.
\end{equation}

  If $l(M)=l(R)r$, then $l(\ker \theta)=0$. Hence $\ker \theta
  =\{0\}$ and $M^\vee \simeq R^r$.  We conclude that $M$ is free as
  well since $M\cong (M^\vee)^\vee\simeq \Hom_R(R^r, R)$. 
\end{proof}

\begin{cor}\label{cor:criterion-free-for-Ap}
  Suppose that $\dim A =r \deg P(T)/2$.  Let $\ell$ be a prime
  distinct from $\Char(K)$, and $\bar{P}(T)=\prod_{i=1}^s
  h_i(T)^{t_i}$ be the factorization of $\bar{P}(T):=P(T)\otimes 1\in
  \ff_\ell[T]=\zz[T]\otimes_\zz \ff_\ell$ into irreducible factors
  over $\ff_\ell$. For each $1\leq i\leq s$, let $W_{\ell, i}:=\{ x\in
  A[\ell]\mid h_i(\phi)x =0\}$.\footnote{A priori, $H(\phi)$ only make
    sense if $H(T)\in \zz[T]$. We may choose $H_i(T)\in \zz[T]$ such
    that its reduction mod $\ell$ is $h_i(T)$. Then for any $x\in
    A[\ell]$, the element $H_i(\phi)x$ does not depends on the choice
    of $H_i(T)$. By an abuse of notation, we will denote this element
    by $h_i(\phi)x$. } Then $A[\ell]$ is a free
  $\ff_\ell[T]/(\bar{P}(T))$-module of rank $r$ if and only if
  $\dim_{\ff_\ell}W_{\ell, i} = r\deg h_i(T)$.
\end{cor}
\begin{proof}
  By the Chinese Reminder Theorem, we may decompose the Artinian ring
  $\ff_\ell[T]/(\bar{P}(T))$ into a direct sum of local Artinian rings:
  \[R:= \ff_\ell[T]/(\bar{P}(T))\cong \bigoplus_{i=1}^s
  \ff_\ell[T]/(h_i(T)^{t_i}).\] Clearly, $\m_i=(h_i(T))$ is the unique
  maximal ideal in $R_i:=\ff_\ell[T]/(h_i(T)^{t_i})$ and its residue
  field is $k_i:=R_i/\m_i\cong \ff_{\ell^{d_i}}$ with $d_i:=\deg
  h_i(T)$. Correspondingly, we have a direct sum decomposition
\[ M:= A[\ell]= \oplus_{i=1}^s M_i,\]
where each $M_i$ is an $R_i$-module. By definition, $W_{\ell, i}$ is
the socle of $M_i$. 

If $M$ is a free $R$-module of rank $r$, then each $M_i$ is a free
$R_i$-module of rank $r$. By Lemma~\ref{lem:gorenstein-ring-free},
$\dim_{\ff_\ell}W_{\ell, i}=[k_i:\ff_\ell]\dim_{k_i}W_{\ell, i} =d_ir$
for all $1\leq i \leq s$.

Now suppose that $\dim_{k_i} W_{\ell, i}=r$ for all $i$. Then by
Lemm~\ref{lem:gorenstein-ring-free}, $ l_{R_i}(M_i)\leq
rl_{R_i}(R_i). $ Hence $M_i\leq \abs{R_i}^r = \ell^{rd_it_i}$.
On the other hand, 
\[ \ell^{r\deg P(T)}=\abs{A[\ell]}=\prod_{i=1}^s \abs{M_i}\leq
\ell^{r\sum_i d_it_i}=\ell^{r\deg P(T)}.\]
So we must have equality at all places. In particular, $l_{R_i}(M_i)=
rl_{R_i}(R_i)$. By Lemm~\ref{lem:gorenstein-ring-free} again, $M_i$ is a
free $R_i$-module of rank $r$ for all $1\leq i\leq s$. Hence $M$ is a
free $R$-module of rank $r$. 
\end{proof}
\begin{lem}\label{lem:free-tate-module-over-ring-of-integers}
  Let $\iota: \OO\hookrightarrow \End(A)$ be an embedding of the ring
  of integers $\OO\subset L$ of a number field $L$ into
  $\End(A)$. Then for any prime $\ell$ distinct from $\Char(K)$,
  $T_\ell A$ is a free $\OO \otimes_\zz \zz_\ell$-module of rank
  $2\dim A/[L:\qq]$.
\end{lem}
\begin{proof}
  This is a well-known fact. By \cite[Theorem 4, p. 180]{MR2514037},
  we have ${\rm Tr}(\iota(a); V_\ell(A))\in \qq$ for all $a\in L$, where
  $V_\ell(A)=T_\ell(A)\otimes_\zz \qq_\ell$. It follows that $V_\ell(A)$
  is a free $L\otimes_\qq \qq_\ell$-module. Since  $\OO \otimes_\zz \zz_\ell$
  is a product of complete discrete valuation rings, the freeness of
  $T_\ell(A)$ follows. 
\end{proof}

\begin{thm}\label{thm:equivalence-tate-free-kernel-connected}
  Suppose that $\dim A= r\deg P(T)/2$, and $\Disc(P(T))$ is coprime to
  $\Char(K)$. Consider the following statements.
  \begin{enumerate}
  \item[(\ref{thm:equivalence-tate-free-kernel-connected}.a)] $ A[p]$
    is a free $\ff_p[T]/(\bar{P}(T))$-module of rank $r$ for any prime
    $p\mid \Disc(P(T))$.
  \item[(\ref{thm:equivalence-tate-free-kernel-connected}.b)] $\ker
    G(\phi)$ is an abelian subvariety of $A$ of dimension $r\deg
    G(T)/2$ for all $G(T)\mid P(T)$.
  \item[(\ref{thm:equivalence-tate-free-kernel-connected}.c)] $\ker
    F(\phi)$ is an abelian subvariety of $A$ of dimension $r\deg
    F(T)/2$ for all irreducible
    $F(T)\mid P(T)$.
  \end{enumerate}
  Then
  {\rm{(\ref{thm:equivalence-tate-free-kernel-connected}.a)}}$\Rightarrow$
  {\rm{(\ref{thm:equivalence-tate-free-kernel-connected}.b)}}
  $\Rightarrow$
  {\rm{(\ref{thm:equivalence-tate-free-kernel-connected}.c)}}. If
  {\rm{(\ref{thm:equivalence-tate-free-kernel-connected}.b)}} holds,
  $\ker G(\phi)= \sum \ker F(\phi)$, where the sum is over all
  irreducible factors $F(T)$ of $G(T)$.  We further assume that
  $\zz[T]/(F(T))$ is a normal integral domain for all irreducible
  factors $F(T)\mid P(T)$. Then
  {\rm{(\ref{thm:equivalence-tate-free-kernel-connected}.a')}}$\Leftrightarrow$
  {\rm{(\ref{thm:equivalence-tate-free-kernel-connected}.b)}}
  $\Leftrightarrow$
  {\rm{(\ref{thm:equivalence-tate-free-kernel-connected}.c)}}, where
  {\rm{(\ref{thm:equivalence-tate-free-kernel-connected}.a')}} is the
  following variant of
  {\rm{(\ref{thm:equivalence-tate-free-kernel-connected}.a)}}:
  \begin{enumerate}
  \item[(\ref{thm:equivalence-tate-free-kernel-connected}.a')] $
    A[\ell]$ is a free $\ff_\ell[T]/(\bar{P}(T))$-module of rank $r$
    for any prime $\ell\neq \Char(K)$.
  \end{enumerate}
\end{thm}
\begin{proof}
  Clearly {\rm{(\ref{thm:equivalence-tate-free-kernel-connected}.b)}}
  $\Rightarrow$
  {\rm{(\ref{thm:equivalence-tate-free-kernel-connected}.c)}}.  We
  prove that
  {\rm{(\ref{thm:equivalence-tate-free-kernel-connected}.a)}}$\Rightarrow$
  {\rm{(\ref{thm:equivalence-tate-free-kernel-connected}.b)}}. Let
  $G'(T)=P(T)/G(T)$, then $\Res(G(T), G'(T))\mid \Disc P(T)$ by
  Subsection~\ref{subsec:resultant-and-generator-of-intersection}. In
  particular, $\Res(G(T), G'(T))$ is coprime to $\Char(K)$. Let
  $\beta:=G(\phi), \gamma:=G'(\phi)\in \End(A)$. Clearly,
  $(\ker\beta)[p]=\{ x\in A[p]\mid \beta x=0\}$.  If $A[p]$ is a free
  $\ff_p[T]/(P(T))$-module of rank $r$, then $\abs{(\ker\beta)[p]}=
  p^{r\deg G(T)}$.  By Lemma~\ref{cor:criterion-for-kernel-connected},
  $\dim \ker \beta\leq r\deg G(T)/2$. Similarly, $\dim \ker \gamma\leq
  r\deg G'(T)/2$.  However, by Lemma~\ref{lem:reduced-kernel},
  $\dim \ker \beta +\dim \ker \gamma=\dim A$. So we have 
\[  \begin{split}
   \dim A&= r\deg P(T)/2= r\deg G(T)/2 + r\deg G'(T)/2 \\&\geq \dim \ker
\beta +\dim \ker\gamma = \dim A. 
  \end{split}
 \]
Therefore, $\dim \ker\beta= r\deg G(T)/2$. We conclude that $\ker
\beta$ is connected by Lemma~\ref{cor:criterion-for-kernel-connected}
again. This proves that   {\rm{(\ref{thm:equivalence-tate-free-kernel-connected}.a)}}$\Rightarrow$
  {\rm{(\ref{thm:equivalence-tate-free-kernel-connected}.b)}}.

  If $G_1(T)$ and $G_2(T)$ are coprime divisors of $P(T)$, then $(\ker
  G_1(\phi)) \cap (\ker G_2(\phi))$ is a finite \'etale group scheme
  over $K$, so
\[\dim(\ker G_1(\phi)+\ker G_2(\phi))=\dim \ker G_1(\phi)+\dim
\ker G_2(\phi). \] Suppose that $G(T)=\prod_{i=1}^t F_i(T)$ with each
$F_i(T)$ irreducible and pairwise distinct. By induction, $\dim \ker
G(\phi)= \dim (\sum_{i=1}^t \ker F_i(\phi))$. Clearly
\begin{equation}
  \label{eq:12}
\ker G(\phi)\supseteq \sum_{i=1}^t \ker F_i(\phi).  
\end{equation}
 Suppose that
(\ref{thm:equivalence-tate-free-kernel-connected}.b) holds. Then both
sides of (\ref{eq:12}) are abelian varieties of the same dimension. So
they must be the same. 

Suppose that $\zz[T]/(F(T))$ is integrally closed for all irreducible
factors $F(T)$ of $P(T)$. To show the statements are equivalent, it is
enough to prove
{\rm{(\ref{thm:equivalence-tate-free-kernel-connected}.c)}}$\Rightarrow$
{\rm{(\ref{thm:equivalence-tate-free-kernel-connected}.a')}}. Suppose
that $\bar{P}(T)= \prod_{i=1}^s h_i(T)^{t_i}$ is the prime
factorization of $\bar{P}(T)$ over $\ff_\ell$.  Let $W_{\ell, i}=\{
x\in A[\ell]\mid h_i(\phi)x=0\}$. By
Corollary~\ref{cor:criterion-free-for-Ap}, $A[\ell]$ is a free
$\ff_\ell[T]/(\bar{P}(T))$-module of rank $r$ if we can prove that
$\dim_{\ff_\ell}W_{\ell, i}= r\deg h_i(T)$ for all $1\leq i \leq s$.

For each fixed $h_i(T)$, there exists an irreducible factor $F(T)$ of
$P(T)$ such that $h_i(T)\mid \bar{F}(T)$. Therefore,
\[W_{\ell, i}=\{ x\in A[\ell]\mid h_i(\phi)x=0\}= \{ x\in \ker
F(\phi)[\ell]\mid h_i(\phi)x=0\}.\] Suppose that $\ker F(\phi)$ is
an abelian subvariety of $A$ of dimension $r\deg F(T)/2$.  There is an
embedding $\zz[T]/(F(T))\hookrightarrow \End(\ker F(\phi))$ given by
$T\mapsto \phi\mid_{\ker F(\phi)}$. Since $\zz[T]/(F(T))$ is a
normal integral domain, it follows from
Lemma~\ref{lem:free-tate-module-p-torsion} and
Lemma~\ref{lem:free-tate-module-over-ring-of-integers} that $(\ker
F(\phi))[\ell]$ is a free $\ff_\ell[T]/(\bar{F}(T))$-module of rank $r$. By
Corollary~\ref{cor:criterion-free-for-Ap}, the $\ff_\ell$-vector space
$\{ x\in \ker F(\phi)[\ell]\mid h_i(\phi)x=0\}$ has dimension $r\deg
h_i(T)$. We obtain the desired result. 
\end{proof}
\section{superelliptic Jacobians}

In this section, we prove the theorems and their corollaries stated in
the introduction. Certain simple arithmetic results  are
postponed to Section~4. We keep the notations and assumptions of
Section~\ref{sec:introduction}. Recall that $K$ is an algebraically
closed field, $C_{f,N}$ is the smooth projective curve over $K$
defined by $Y^N=f(X)=\prod_{i=1}^n(X-\alpha_i)^{e_i}$ with $f(x)$
satisfying the conditions in (\ref{eq:assumption-on-f}), and
$J_N:=\Jac(C_{f,N})$ is the Jacobian of $C_{f,N}$.  There is a natural
action of $\rho_J: \bmu_N\to \Aut(J_N)$ of $\bmu_N\subset K^\times$ on
$J_N$, and $\rho_J(\bmu_N)=\dangle{\delta_N}$.

\begin{sect}
  The assumptions in (\ref{eq:assumption-on-f}) guarantee that there
  is exactly one point in $C_{f,N}(K)$ corresponding to the point
  $(\alpha_i,0)$ on the affine curve $Y^N=f(X)$, and moreover, there
  is a unique point (denoted by $\infty$) in $C_{f,N}(K)$ that lies
  above the point at infinity on $\pp^1(K)$ for the map $\eta:
  C_{f,N}\to \pp^1$. Clearly, $\bmu_N$ fixes the following set of
  points
\begin{equation}
  \label{eq:fixed-points-on-curve}
  \FV(C_{f,N}):= \{ \QQ_1=(\alpha_1, 0), \cdots, \QQ_n=(\alpha_n, 0),
\infty\}\subseteq C_{f,N}(K),
\end{equation}
and it acts freely outside $\FV(C_{f,N})$. Therefore $\eta:C_{f,N}\to
\pp^1$ is totally ramified at each point of $\FV(C_{f,N})$ with
ramification index $N$, and unramified everywhere else. All the
ramifications are tame since the characteristic of $K$ does not divide
$N$.  By the Hurwitz formula \cite[Corollary IV.2.4]{MR0463157}, the
genus of $C_{f,N}$ is (cf. \cite{MR1107394} for the case $K=\cc$)
\begin{equation}
  \label{eq:genus}
   g(C_{f,N})=\frac{(N-1)(n-1)}{2}.
\end{equation}
\end{sect}

\begin{sect}\label{subsec:description-of-the-obvious-part}
  A natural question is to describe the group of all fixed points of
  $\bmu_N$ on $J_N$. Let us denote it by
\[ \FV_N:=(J_N)^{\bmu_N}=\{ x\in J_N(K)\mid \delta_N  x=x\}.\]
It contains an obvious subgroup consisting of the linear equivalence classes of divisors of degree
zero supported on $\FV(C_{f,N})$:
\[\GV_N:=\left\{[\DD]\in \Pic^0(C_{f,N})=J_N(K)\;\middle |\; \DD=b\infty+\sum_{i=1}^n a_i\QQ_i, \quad
  \deg \DD =b+\sum_{i=1}^na_i=0\right\}. \]

We will describe the group structure of $\GV_N$. Given a rational
function $g\in K(C_{f,N})$ on $C_{f,N}$, let $\Div(g)$ be its
divisor. Then
\begin{align*}
  \Div (Y) &= \sum_{i=1}^n e_i \QQ_i - \deg(f)\infty,\\
  \Div (X-\alpha_i)&= N \QQ_i- N\infty. 
\end{align*}

Since $\gcd(\deg(f),N)=1$, we may find $a, b \in \zz$ such that
$a\deg(f)+bN=1$. Then 
\[ \Div (Y^a(X-\alpha_i)^b)= a\sum_{j=1}^n e_j\QQ_j + bN\QQ_i-
(a\deg(f)+bN)\infty=a\sum_{j=1}^n e_j\QQ_j + bN\QQ_i- \infty. \]
Therefore, any divisor of degree zero supported on $\FV(C_{f,N})$ is linear
equivalent to one supported on the set 
\[ \RV:=\{ \QQ_1, \cdots, \QQ_n\}.\] By \cite[Lemma 4.1]{xue-yu}, a
divisor of degree zero of the form $\DD=\sum_{i=1}^n a_i \QQ_i$ is
linear equivalent to zero if and only if there exists $c\in \zz$ such
that $a_i\equiv ce_i \pmod{N}$ for all $1\leq i\leq n$. For this
$c\in\zz$, we have $0=\sum_{i=1}^n a_i \equiv c\sum_{i=1}^n e_i
\pmod{N}$. Since $\deg f(x)=\sum_{i=1}^n e_i$ is coprime to $N$,
$c\equiv 0 \pmod{N}$. In other words, $\DD$ is linear equivalent to
zero if and only if $a_i\equiv 0 \pmod{N}$.

Let $M_\RV$ be the free $(\zmod{N})$-module of rank $n$ generated by
elements of $\RV$, and 
\begin{gather*}
  M_\RV^0:=\left \{ m\in M_\RV\;\middle |\;  m=\sum_{i=1}^n a_i \QQ_i,  \;
    a_i\in \zmod{N}, \;
  \sum_{i=1}^na_i=0.\right\},\\
\EE_0:=\sum_{i=1}^n e_i\QQ_i\in M_\RV. 
\end{gather*}
Then $M_\RV= M_\RV^0\oplus (\zmod{N}) \EE_0$ and $
M_\RV^0\cong (\zmod{N})^{n-1}$. We have a canonical isomorphism
\begin{equation}
  \label{eq:1}
  \GV_N\cong M_\RV^0\cong M_\RV/((\zmod{N})\EE_0).
\end{equation}
Our first goal in this section is to show that $\FV_N=\GV_N$. 
\end{sect}

\begin{sect}\label{subsec:calculation-of-character}
  We refer to \cite[Section VI.2]{Serre_local} for the notation
  $\mathbbm{a}_\QQ$ below. It is the character of the Artin
  representation of $\bmu_N$ at $\QQ \in C_{f,N}(K)$ which encodes the
  ramification information at each point $\QQ$ for the map $\eta:
  C_{f,N}\to \pp^1=C_{f,N}/\bmu_N$. If $\eta$ is unramified at $\QQ$,
  then $\mathbbm{a}_\QQ(\xi)=0$ for all $\xi\in \bmu_N$. If $\eta$ is
  totally ramified at $\QQ$, $\mathbbm{a}_\QQ$ may be defined in the
  following way (Combining \cite[Lemma III.6.3]{Serre_local} and
  \cite[Section IV.1]{Serre_local}). Let $\pi_\QQ\in
  K(C_{f,N})^\times$ be a local parameter at $\QQ$, and
  $\mathrm{val}_\QQ: K(C_{f,N})^\times \to \zz$ be the valuation of
  $K(C_{f,N})$ associated to $\QQ$. Then $\forall\, \xi\in \bmu_N$,
\[ \mathbbm{a}_\QQ(\xi)= -\mathrm{val}_\QQ\left(\rho_C(\xi)\pi_\QQ-\pi_\QQ\right)\quad
\text{ if } \xi\neq 1, \qquad \mathbbm{a}_\QQ(1)=-\sum_{\xi\in \bmu_N,
  \xi\neq 1} \mathbbm{a}_\QQ(\xi).\]
For all $\QQ'\in \pp^1(K)$, $\mathbbm{a}_{\QQ'}$ is defined to be
$\sum_{\QQ\mapsto \QQ'} \mathbbm{a}_\QQ$. 

Let us fixed a prime $\ell\neq \Char K$.  Since $\eta$ is totally and
tamely ramified at each point $\QQ\in \FV(C_{f,N})$, we
have \[\mathbbm{a}_{\eta(\QQ)}=\mathbbm{a}_\QQ=
\mathbbm{r}_{\bmu_N}-\mathbbm{1}_{\bmu_N}=\mathbbm{u}_{\bmu_N},\]
where $\mathbbm{r}_{\bmu_N}, \mathbbm{1}_{\bmu_N},
\mathbbm{u}_{\bmu_N}:\bmu_N\to \qq_\ell$ are the characters of the
regular representation, the 1-dimensional trivial representation, and
the augmentation representation of $\bmu_N$ respectively. More
precisely, $\mathbbm{1}_{\bmu_N}(\xi)=1, \forall \,\xi\in \bmu_N$;
$\mathbbm{r}_{\bmu_N}(1)=N$, and $\mathbbm{r}_{\bmu_N}(\xi)=0$ for all
$\xi\neq 1$. Hence
  \[\forall\, \xi\in \bmu_N, \qquad \mathbbm{u}_{\bmu_N}(\xi)=
\begin{cases}
  N-1   &\qquad \text{if } \xi=1,\\
  -1   &\qquad \text{if } \xi\neq 1.\\
\end{cases}
\]
Let $\mathbbm{h}_1: \bmu_N\to \qq_\ell$ be the character of the
representation of $\bmu_N$ defined by $T_\ell J_N$.  By \cite[Section
VI.4]{Serre_local}, we have
\[
\begin{split}
  \mathbbm{h}_1&= \sum_{\QQ\in \FV(C_{f,N})} \mathbbm{a}_{\eta(\QQ)}
  +2\cdot \mathbbm{1}_{\bmu_N} -
  E(\pp^1)\cdot \mathbbm{r}_{\bmu_N}.\\
  &=
  (n-1)(\mathbbm{r}_{\bmu_N}-\mathbbm{1}_{\bmu_N})=(n-1)\mathbbm{u}_{\bmu_N}.
\end{split}\]
Here $E(\pp^1)=2$ is the Euler characteristic of $\pp^1$. 
\end{sect}

\begin{sect}\label{subsec:characteristic-poly-of-deltaN}
  Since $\rho_J(\bmu_N)$ is generated by $\delta_N$, $\FV_N=(J_N)^{\bmu_N}=
  \ker(1-\delta_N)$, so
  \[\abs{\FV_N}=\abs{\ker(1-\delta_N)}=\deg(1-\delta_N).\]
  By \cite[Theorem 19.4]{MR2514037}, $\deg (1-\delta_N)= \det
  T_\ell(1-\delta_N)$. We may choose the prime $\ell$ such that $\ell
  \equiv 1 \pmod{N}$. Then $\ell$ splits completely in $\zz[\zeta_N]$,
  and $\mathbbm{u}_{\bmu_N}= \sum_{\chi\neq 1} \chi$, where the sum is
  over all nontrivial characters $\chi: \bmu_N\to \qq_\ell^\times$.
  It follows from Subsection~\ref{subsec:calculation-of-character}
  that the characteristic polynomial of $T_\ell(\delta_N)$ is
  \[ \det (T-T_\ell(\delta_N))=
  \left(\prod_{i=1}^{N-1}(T-(\zeta_N)^i)\right)^{n-1}= P_N(T)^{n-1},\]
  where $P_N(T)$ is given by (\ref{eq:def-of-PN}), and the minimal
  polynomial of $T_\ell(\delta_N)$ is $P_N(T)$.  Since the natural map
  $\End(J_N)\otimes_\zz \zz_\ell\to \End_{\zz_{\ell}}(T_\ell J_N)$ is
  an embedding \cite[Theorem 19.3]{MR2514037}, the minimal polynomial
  of $\delta_N$ is equal to $P_N(T)$.  We have
  \[\deg(1-\delta_N)=P_N(1)^{n-1}=N^{n-1}.\] 
  Recall that $\FV_N$ contains the subgroup $\GV_N\simeq
  (\zmod{N})^{n-1}$ by (\ref{eq:1}). They must coincide by comparing
  the cardinality. We have proven the following theorem:
\end{sect}
\begin{thm}\label{thm:invariant-theorem}
We have  $\FV_N=\GV_N\simeq (\zmod{N})^{n-1}$.
\end{thm}

For the case $N=p$ is a prime, Theorem~\ref{thm:invariant-theorem} was
already contained in \cite[Section 6]{MR1465369} and \cite[Proposition
3.2]{MR1612262}.

\begin{sect}\label{subsec:properties-of-zT-mod-PN}
  Since $P_N(T)=\prod_{D\mid N, D>1} \Phi_D(T)$, by the Chinese
  Remainder Theorem,
\begin{equation}
  \label{eq:prod-cyclotomic-field}
 \qq[T]/(P_N(T)) \cong \prod_{D\mid N, D>1 } \qq[T]/(\Phi_D(T))
 \cong \prod_{D\mid N, D>1 } \qq[\zeta_D].  
\end{equation}
On the other
hand, it is important to note that if $N$ is not prime, the embedding
\begin{equation}
  \label{eq:9}
\zz[T]/(P_N(T))\hookrightarrow \prod_{D\mid N, D>1 }
\zz[T]/(\Phi_D(T)).
\end{equation}
is \textit{not} an isomorphism.  For example, if $N=p^r$ for some
$r>1$, then
\[ \left(\zz[T]/(P_N(T))\right)\otimes_\zz \ff_p \cong
\ff_p[T]/(P_N(T)) = \ff_p[T]/((T-1)^{N-1}). \] The right hand side is
a local ring, and therefore not a direct product of proper subrings.
We leave it to the reader to prove that (\ref{eq:9}) is not an
isomorphism for arbitrary $N$ not a prime. However, from an explicit
construction (cf. \cite[Lemma 5.2]{MR1708603} or \cite[Subsection
2.6]{xue-yu}), one may show that the
idempotents in $\qq[T]/(P_N(T))$ lie in $\frac{1}{N}\zz[T]/(P_N(T))$.
Therefore,
\[\frac{1}{N}\zz[T]/(P_N(T))\supset \prod_{D\mid N, D>1 }
\zz[T]/(\Phi_D(T)),\] and the cokernel of (\ref{eq:9}) are
$N$-torsions. There is an isomorphism
\begin{equation}
  \label{eq:6}
 \zz[1/N, T]/(P_N(T)) \cong \prod_{D\mid N, D>1 }
\zz[1/N, T]/(\Phi_D(T)). 
\end{equation}
We leave it to the reader to show that
\begin{equation}
  \label{eq:14}
\Disc(P_N(T))=(-1)^{(N-1)(N-2)/2}N^{N-2}.
\end{equation}
\end{sect}

% \begin{lem}
%   $\Disc(P_N(T))=
% (-1)^{(N-1)(N-2)/2}N^{N-2}.$
% \end{lem}
% \begin{proof}
% \[\Disc(P_N(T))\cdot \Res(T-1, P_N(T))^2= \Disc(T^N-1). \]
% We have $\Res(T-1, P_N(T))=P_N(1)^2=N^2$. On the other hand 
% \[
% \begin{split}
% \Disc(T^N-1)&= (-1)^{N(N-1)/2}\prod_{\zeta^N=\zeta'^N=1, \zeta\neq
%   \zeta'}(\zeta-\zeta')\\
%     &=(-1)^{N(N-1)/2} \prod_{\zeta^N=1} \left(\zeta^{N-1}\prod_{\zeta'^N=1,
%       \zeta'\neq \zeta}(1-\zeta'/\zeta) \right)\\
%     &=  (-1)^{N(N-1)/2}N^N \prod_{\zeta^N=1} \zeta^{-1}=
%     (-1)^{N(N-1)/2}N^N\cdot (-1)^{N-1}\\
%     &=(-1)^{(N-1)(N-2)/2}N^N.
% \end{split}
% \]
% \end{proof}

\begin{prop}\label{prop:free-away-from-N}
  For any prime $\ell \nmid (N\Char(K))$, the Tate module $T_\ell J_N$
  is a free $\zz_\ell[T]/(P_N(T))$-module of rank $n-1$.
\end{prop}
\begin{proof}
  Since $\ell \nmid N$, $\zz_\ell[T]/(P_N(T))$ is a product of
  discrete valuation rings by (\ref{eq:6}).  Because $T_\ell(J_N)$ is
  $\zz_\ell$-torsion free, it is enough to prove that
  $V_\ell(J_N):=T_\ell(J_N)\otimes_{\zz_\ell}\qq_\ell$ is a free
  $\qq_\ell[T]/(P_N(T))$-module of rank $n-1$. This follows directly
  from Subsection~\ref{subsec:calculation-of-character}, noting that
  the representation space of the augmentation representation
  $\mathbbm{u}_{\bmu_N}$ over $\qq_\ell$ is isomorphic to
  $\qq_\ell[T]/(P_N(T))$.
\end{proof}

\begin{sect}\label{subsec:basic-properties-maps}
  For each integer $D\mid N$ and $D>1$, there exists a finite morphism
\[ \eta_D: C_{f,N}\to C_{f,D}, \qquad (X,Y)\mapsto (X, Y^{N/D}).\]
It induces two maps between the Jacobians 
\[ \eta_D: J_N\to J_D,\qquad \eta_D^*: J_D\to J_N,\] by Albanese
functoriality and Picard functoriality respectively.  We describe the
action of these maps on closed points.  Recall that
$J_D(K)=\Div^0(C_{f,D})/\mathord{\sim}$\,, the group of divisors of
degree zero modulo linear equivalence. Given a divisor $\DD\in
\Div(C_{f,D})$ on $C_{f,D}$, we write $[\DD]$ for the linear
equivalence class of $\DD$.  The abelian group $J_D(K)$ is generated
by the set of elements $\{[\QQ-\infty_D]\}_{\QQ\in C_{f,D}(K)}$, where
$\infty_D$ is the unique point at infinity on $C_{f,D}$. Let $\nu_D:
C_{f,D}\to J_D$ be the closed immersion defined by $\QQ\mapsto
[\QQ-\infty_D]$ (cf. \cite[Section 2]{MR861976}). Then by definition,
$\eta_D: J_N\to J_D$ is the unique homomorphism such that the
following commutative diagram holds (cf. \cite[Proposition 6.1]{MR861976})
\[
\begin{CD}
  C_{f,N}@>{\eta_D}>>  C_{f,D}\\
  @V{\nu_N}VV   @VV{\nu_D}V\\
  J_N @>{\eta_D}>>  J_D
\end{CD}
\]
It follows that 
\begin{equation}
  \label{eq:albanese-map}
  \eta_D([\QQ-\infty_N])=
[\eta_D(\QQ)-\eta_D(\infty_N)]=[\eta_D(\QQ)-\infty_D] \qquad \forall\,
\QQ\in C_{f,N}(K).
\end{equation}
On the other hand, let $M=N/D$, then $\eta_D:C_{f,N}\to C_{f,D}$
realizes $C_{f,D}$ as a quotient of $C_{f,N}$ by the group
$\bmu_M:=\dangle{(\xi_N)^D}\subseteq \bmu_N$. For each $\QQ\in
C_{f,N}(K)$, we write $e_\QQ$ for the ramification index of $\eta_D$
at $\QQ$. Then $\forall\, \QQ'\in C_{f,D}(K)$, 
\begin{equation}
  \label{eq:picard-map}
   \eta_D^*([\QQ'-\infty_D])= \left[\sum_{\QQ\mapsto
    \QQ'}e_{\QQ}\QQ-M\infty_N\right]= \left[\sum_{\xi\in \bmu_M}
  \rho_J(\xi)\QQ_0 -M\infty_N\right] ,
\end{equation}
where $\QQ_0\in C_{f,N}(K)$ is a fixed point in $\eta_D^{-1}(\QQ')$.

Combining (\ref{eq:albanese-map}) and (\ref{eq:picard-map}), we see
that 
\[\eta_D\circ \eta_D^* =\deg(\eta_D)\cdot \Id_{J_D}= M\cdot
\Id_{J_D}.\] Moreover, the composition of $ J_N\xrightarrow{\eta_D}
J_D \xrightarrow{\eta_D^*} J_N$ is given by
\[[\DD]\mapsto \sum_{\xi\in \bmu_M}\rho_J(\xi) [\DD] =Q_{N,D}(\delta_N)([\DD]), \qquad
\forall \; [\DD]\in J_N(K),\] where
\begin{equation}
  \label{eq:10}
Q_{N,D}(T):=\frac{T^N-1}{T^D-1}=\sum_{i=0}^{\frac{N}{D}-1}T^{iD}
  =\prod_{D'\mid N, D'\nmid D} \Phi_{D'}(T)\in \zz[T].
\end{equation}

Clearly $\delta_N^*=\delta_N^{-1}\in \End(J_N)$. Let
$\xi_D:=\xi_N^{N/D}\in \bmu_D\subset K^\times$ and define $\delta_D:
J_D\to J_D$ similarly to $\delta_N$.  Then $\delta_D\eta_D=\eta_D
\delta_N$ and $\eta_D^*\delta_D= \delta_N\eta_D^*$. Both $\eta_D$ and
$\eta_D^*$ are $\bmu_N$-equivariant if we let $\bmu_N$ act on $J_D$
via the map $\bmu_N\to \bmu_D$ that sends $\xi_N$ to $\xi_D$.
\end{sect}

\begin{sect}\label{subsec:canonical-principal-polarization}
  Let $\lambda_N:J_N\xrightarrow{\cong} J_N^\vee$ be the canonical
  principal polarization of $J_N$. Under the canonical identification
  $J_N= (J_N^\vee)^\vee$, $\lambda_N^\vee=\lambda_N$.  It induces a
  Rosati involution on $\End(J_N)$ defined by $\phi\mapsto
  \phi':=\lambda_N^{-1}\circ\phi^\vee\circ \lambda_N, \forall\,
  \phi\in \End(J_N)$. The polarization $\lambda_N$ is
  $\rho_J(\bmu_N)$-invariant. For any $\ell\neq \Char(K)$, let
  $E^{\lambda_N}: T_\ell J_N\times T_\ell J_N \to
  \zz_\ell(1):=\varprojlim_{i\in \nn} \bmu_{\ell^i}$ be the
  nondegenerate Riemann form (See \cite[Section 20]{MR2514037})
  defined by $\lambda_N$. Then \[ E^{\lambda_N}(T_\ell(\delta_N) x,
  T_\ell(\delta_N) y)= E^{\lambda_N}(x, y), \qquad \forall x, y\in
  T_\ell J_N.\] In particular,
  $\delta_N'=\delta_N^{-1}=(\delta_N)^{N-1}$, and
  $\zz[\delta_N]\subseteq \End(J_N)$ is invariant under the Rosati
  involution. For any $D\mid N$ with $D>1$, we have
  (cf. \cite[Proposition 11.11.6]{MR2062673} or \cite[Section
  17.5]{MR1987784} in the case $K=\cc$, and \cite[Proposition
  A.6]{MR2987306} in general)
  \begin{equation}
    \label{eq:polarization-albanese-picard}
\eta_D=\lambda_D^{-1}\circ (\eta_D^*)^\vee \circ \lambda_N. 
  \end{equation}
\end{sect}

We refer to \cite[Subsection 2.11]{xue-yu} for the following proposition.
\begin{prop}\label{prop:def-D-new-part}
  The map $\eta_M^*: J_M\to J_N$ is a closed immersion for all $M\mid
  N$ and $M>1$.
\end{prop}

Since $P_M(\delta_M)J_M=\{0\}$, one has $P_M(\delta_N)(\eta_M^*J_M)=
\eta_M^*P_M(\delta_M)J_M=\{0\}$.  We prove that the image
$\eta_M^*J_M$ is in fact uniquely characterized as a subvariety of
$J_N$ by this property.

\begin{prop}\label{prop:characterization-of-JM}
  For each integer $M\mid N$ and $M> 1$, the kernel of
  $P_M(\delta_N): J_N\to J_N$ is $\eta_M^*(J_M)$.
\end{prop}
\begin{proof}
  Let
  $\beta_M:=P_M(\delta_N)=\sum_{i=0}^{M-1}(\delta_N)^i\in \End(J_N)$. Since
  $P_N(T)$ is separable in $K[T]$, $\ker\beta_M$ is reduced by
  Lemma~\ref{lem:reduced-kernel}. As remarked, $\eta_M^*(J_M)\subseteq
  \ker \beta_M$.
  
  For any divisor $\DD\in \Div(C_{f,N})$, we write $\LZ(\DD)$ for the
  invertible sheaf on $C_{f,N}$ associated to $\DD$ \cite[Section
  II.6, p144]{MR0463157}, \[L(\DD):=H^0(C_{f,N}, \LZ(\DD))=\{g\in
  K(C_{f,N})\mid \Div(g)+\DD\geq 0\},\] and $l(\DD):=\dim_K
  L(\DD)$. If $x=[\DD]\in J_N(K)$ is nonzero, then $l(\DD)=0$.  By the
  Riemann-Roch Theorem \cite[Theorem IV.1.3]{MR0463157},
  $l(\DD+t\infty)= t+1-g(C_{f,N})>0$ if $t\in \nn$ is large
  enough. Since $l(\DD+(t+1)\infty)-l(\DD+t\infty)\leq 1$ for all
  $t\geq 0$, there exists a smallest $t$ such that
  $l(\DD+t\infty)=1$. In other words, for this $t$ there exists a
  unique effective divisor $\DD_0>0$ such that $\DD+t\infty \sim
  \DD_0$. Clearly, $\DD_0$ depends only on $[\DD]$. The coefficient
  $b_\infty$ in $\DD_0=\sum_{\PP\in C_{f,N}(K)} b_\PP \PP$ is
  necessarily zero by the minimality of $t$.

  Any point $x=[\DD]\in \ker \beta_M$ is fixed by
  $(\delta_N)^M$. Choose $t$ and $\DD_0$ for $[\DD]$ as above. By the
  uniqueness of $\DD_0$, we must have $(\delta_N)^M\DD_0= \DD_0$
  (Equality of divisors). Let $\DD_0= \DD_0'+\DD_0''$, with $\DD_0'=
  \sum_{i=1}^n b_i \QQ_i$ and the support of $\DD_0''$ disjoint from
  $\FV(C_{f,N})=\{\QQ_1, \cdots, \QQ_n, \infty\}$. We write $t_1=\deg
  \DD_0'$ and $t_2=\deg \DD_0''$, then $t_1+t_2=t$ and $\DD\sim
  (\DD_0'-t_1\infty) + (\DD_0''-t_2\infty)$. Clearly $\DD_0''$ is
  fixed by $(\delta_N)^M$. In other words, if $\PP\in \supp \DD_0''$,
  then $\PP'\in \supp \DD_0''$ for all $\PP'\in \eta_M^{-1}
  (\eta_M(\PP))=\{ \PP, (\delta_N)^M \PP, \cdots,
  (\delta_N)^{N-M}\PP\}$.  Therefore, $y:=[\DD_0''-t_2\infty]\in
  \eta_M^*(J_M)(K)$.

  Now we have $z:=[\DD_0'-t_1\infty]=[\DD]-[\DD_0''-t_2\infty]\in
  \ker\beta_M$. By construction $\delta_N z=z$. So $\beta_M z=
  \sum_{i=0}^{M-1} (\delta_N)^i z =M z$, and hence $z\in \FV_N[M]$. We
  claim that $\FV_N[M]=\eta_M^*(\FV_M)$. Indeed, by
  Theorem~\ref{thm:invariant-theorem}, $\FV_N\simeq  (\zmod{N})^{n-1}$, so $\FV_N[M]\cong (\zmod{M})^{n-1}$. On the other
  hand, $\eta_M^*: J_M\to J_N$ is a closed immersion by
  Proposition~\ref{prop:def-D-new-part}, so
  \[ (\zmod{M})^{n-1}\simeq \eta_M^*(\FV_M)\subseteq \FV_N[M]\simeq
  (\zmod{M})^{n-1}.\] It follows that $z\in \eta_M^*(\FV_M)\subset
  \eta_M^*(J_M)(K)$. So $x=y+z\in \eta_M^*(J_M)(K)$, and $\ker \beta_M=
  \eta_M^*(J_M)$.
\end{proof}

\begin{proof}[Proof of the main theorem by induction]
  If $N=\ell$ is a prime, then $\zz[T]/(P_\ell(T))=\zz[T]/(\Phi_\ell(T))$ is
  isomorphic to the ring of integers $\zz[\zeta_\ell]$ in the cyclotomic
  field $\qq(\zeta_\ell)$. The theorem follows from
  Lemma~\ref{lem:free-tate-module-over-ring-of-integers}.

  Suppose that the theorem holds for all $J_D$ with $D\mid N$ and
  $D\neq N$.  The case $\ell\nmid N\Char(K)$ is already treated in
  Proposition~\ref{prop:free-away-from-N}. Now fix a prime $\ell\mid
  N$. By Lemma~\ref{lem:free-tate-module-p-torsion}, it is enough to
  prove that $J_N[\ell]$ is a free $\ff_\ell[T]/(\bar{P}_N(T))$-module
  of rank $n-1$.  Let $N=qM$ with $q=\ell^r$ for some $r\in \nn$ and
  $\gcd(q,M)=1$.  In $\ff_\ell[T]$, $\bar{P}_N(T)=(T^N-1)/(T-1)$
  factorizes as
  \begin{equation}
    \label{eq:factorization-of-Pn-mod-p}
        \frac{T^N-1}{T-1}=\frac{(T^M-1)^q}{T-1}=(T-1)^{q-1}\bar{P}_M(T)^q=
    (T-1)^{q-1} \prod_{i=1}^s h_i(T)^q,    
  \end{equation}
  where each $h_i(T)$ is a monic irreducible factor of
  $\bar{P}_M(T)$. Because $\gcd(M,\ell)=1$, $\bar{P}_M(T)$ is separable
  over $\ff_\ell$, so all $h_i(T)$ in
  (\ref{eq:factorization-of-Pn-mod-p}) are distinct. By
  Theorem~\ref{thm:invariant-theorem}, 
\[W_{\ell, 0}:=\{ x\in J_N[\ell]\mid (\delta_N-1)x=0\}=\FV_N[\ell]\cong
(\zmod{\ell})^{n-1}.\]
On the other hand, by Proposition~\ref{prop:characterization-of-JM}, 
\[ W_{\ell, i}=\{x \in J_N[\ell]\mid h_i(\delta_N)x=0\}\subset \eta_M^*
J_M[\ell].\] By the induction hypothesis, $J_M[\ell]$ is a free
$\ff_\ell[T]/(\bar{P}_M(T))$-module of rank $n-1$. So $\dim_{\ff_\ell} W_{\ell,
  i}= (n-1)\deg h_i(T)$ by
Corollary~\ref{cor:criterion-free-for-Ap}. Applying the same Corollary
again, we see that $J_N[\ell]$ is a free $\ff_\ell[T]/(\bar{P}_N(T))$-module
of rank $n-1$. Therefore, $T_\ell J_N$ is a free
$\zz_\ell[T]/(P_N(T))$-module of rank $n-1$ by
Lemma~\ref{lem:free-tate-module-p-torsion}. 
\end{proof}
\begin{proof}[Proof of Corollary~\ref{cor:optimal-embedding}]
  Recall that we have an embedding of $\qq$-algebras
  \[ \iota: E:=\qq[T]/(P_N(T))=(\zz[T]/(P_N(T)))\otimes_\zz
  \qq\hookrightarrow \End^0(J_N),\] and we want to show that
  $E\cap \End(J_N) = \zz[T]/(P_N(T))$, where the intersection is taken
  within $\End^0(J_N)$. Since $R_\ell:=\zz_\ell[T]/(P_N(T))$ is the
  maximal order in $E_\ell:=E\otimes_\qq \qq_\ell$ for any prime
  $\ell\nmid N$, it is enough to prove that for all $p\mid N$,
  \[ R_p:=\zz_p[T]/(P_N(T)) = E_p \cap (\End (J_N)\otimes_\zz \zz_p)
  \quad \text{inside} \quad \End^0(J_N)\otimes_\qq \qq_p.\] By
  \cite[Theorem 19.3]{MR2514037}, $E_p\cap (\End(J_N)\otimes_\zz
  \zz_p) \subseteq \End_{R_p}(T_p(J_N))$. So it reduces to prove that
  \[ R_p = E_p \cap \End_{R_p}(T_p(J_N)) \quad \text{inside}
  \quad \End_{E_p}(V_p(J_N))=\End_{R_p}(T_p(J_N))\otimes_{\zz_p}\qq_p.\]
  Now by Theorem~\ref{thm:main}, $\End_{R_p}(T_p(J_N))\simeq
  \Mat_{n-1}(R_p)$ and hence $\End_{E_p}(V_p(J_N))\simeq
  \Mat_{n-1}(E_p)$. The embedding $\iota\otimes \qq_p$ identifies
  $E_p$ with the scalar matrices $E_p\cdot\Id$. Clearly $ E_p\cdot \Id
  \cap \Mat_{n-1}(R_p) =R_p\cdot \Id$.
\end{proof}

We assume that $K$ is \textit{not} algebraically closed exclusively for the
following theorem. Let $\bar{K}$ be a fixed algebraic closure of $K$.

\begin{thm}
  Let $K$ be a field of characteristic zero, and $f(X)\in K[X]$ be a
  polynomial with no multiple roots and $n=\deg f\geq 5$. Suppose $p$ is
  a prime that does not divide $n$, Let $r>1$ be a positive integer
  and $q=p^r$. Assume also that either $n=q+1$ or $q$ does no divides
  $n-1$. If $p=2$ then we assume addditionally that $n=kq+c$ with
  nonnegative integers $k$ and $c<q$ such that either $k$ is odd or
  $c<q/2$. Suppose $\Gal(f)$ contains a doubly transitive simple
  non-abelian subgroup $\mathcal{G}$. Then
  $\End_{\bar{K}}(\Jac(C_{f,q}))\simeq \zz[T]/(P_q(T))$.
\end{thm}
\begin{proof}
  By \cite[Corollary 5.4]{MR2471095},
  $\End_{\bar{K}}^0(\Jac(C_{f,q}))=\qq[\delta_q]\simeq
  \qq[T]/(P_q(T))$ under the above assumptions.  Now the theorem
  follows directly from Corollary~\ref{cor:optimal-embedding}.
\end{proof}

The rest of this section is devoted to the proof of
Theorem~\ref{thm:results-on-JN-new}. 
\begin{sect}\label{subsec:def-old-and-new-part}
  For each $D\mid N$ and $D>1$, let
  $\gamma_D:=\Phi_D(\delta_N)\in \End(J_N)$.  By
  Theorem~\ref{thm:equivalence-tate-free-kernel-connected},
  $\ker\gamma_D$ is an abelian subvariety of $J_N$ of dimension
  $\varphi(D)(n-1)/2$. We give a more geometric description of these
  subvarieties.  Suppose that $D_1\mid D_2\mid N$ with
  $D_i>1$, then the map $\eta_{D_1}: C_{f,N}\to C_{f,D_1}$ factors as
  $C_{f,N}\to C_{f,D_2}\to C_{f,D_1}$. By functoriality,
  $\eta_{D_1}^*: J_{D_1}\hookrightarrow J_N$ factors as
  $J_{D_1}\hookrightarrow J_{D_2}\hookrightarrow J_N$. In particular,
  $\eta_{D_1}^*J_{D_1}$ is a subvariety of $\eta_{D_2}^*J_{D_2}$
  inside $J_N$. Following \cite[Section 5]{MR1708603}, we define
  \[ J_N^{\mathrm{old}}:= \sum_{\substack{D\mid N,\\ 1<D<N}} \eta_D^*
  J_D= \sum_{p\mid N} \,\eta_{(N/p)}^*\,J_{(N/p)}.\] The orthogonal
  complement of $J_N^{\mathrm{old}}$ with respect to the canonical
  principal polarization $\lambda_N$ is called the new part of the
  Jacobian and denoted by $J_N^\nw$.  We write $\epsilon_N:J_N^\nw \to
  J_N$ for the canonical embedding. If $N=p$ is a prime, $J_p^\nw$ is
  defined to be $J_p=\Jac(C_{f,p})$.

  Let $G(T)=P_N(T)/\Phi_N(T)\in \zz[T]$. By
  Theorem~\ref{thm:equivalence-tate-free-kernel-connected} and
  Proposition~\ref{prop:characterization-of-JM},
\[ J_N^{\mathrm{old}}= \sum_{D\mid N, 1<D<N}
  \ker \gamma_D.\]
Therefore, $J_N^{\mathrm{old}}= \ker G(\delta_N)$.  In particular,
\[\dim J_N^{\mathrm{old}}= (n-1)\deg G(T)/2=(n-1)(N-1-\varphi(N))/2,\]
and $\dim J_N^\nw =\dim J_N-\dim J_N^{\mathrm{old}}=
\varphi(N)(n-1)/2$ (cf. also \cite[Corollary 5.4]{MR1708603}). 

The map $\sum \eta_{N/p}^*: \prod_{p\mid N} J_{(N/p)} \to J_N$ factors
as
\[\prod_{p\mid N} J_{(N/p)} \xrightarrow{\pi} J_N^{\mathrm{old}}
\xrightarrow{\jmath} J_N.  \] The orthogonal complement of
$J_N^{\mathrm{old}}$ with respect to $\lambda_N$ is defined to be the
identity component (with the reduced subscheme structure) of the map
(See \cite[Theorem 19.1]{MR2514037})
\begin{equation}
  \label{eq:15}
 J_N\xrightarrow{\lambda_N} J_N^\vee \xrightarrow{\jmath^\vee}
 (J_N^{\mathrm{old}})^\vee.  
\end{equation}

Now compose the map in (\ref{eq:15}) with
\begin{equation}
  \label{eq:16}
 (J_N^{\mathrm{old}})^\vee \xrightarrow{\pi^\vee}\prod_{p\mid N}
J_{(N/p)}^\vee \xrightarrow{\prod_{p\mid N}\lambda_{(N/p)}^{-1}}
\prod_{p\mid N} J_{(N/p)}\xrightarrow{\prod_{p\mid N}\eta_{(N/p)}^*} \prod_{p\mid N}J_N  
\end{equation}
 and then apply (\ref{eq:polarization-albanese-picard}), we obtain the map 
 \begin{equation}
   \label{eq:17}
\prod_{p\mid N}Q_{N,N/p}(\delta_N): J_N\to \prod_{p\mid N} J_N,    
 \end{equation}
 where $Q_{N,N/p}(T)\in \zz[T]$ is defined in (\ref{eq:10}). By
 Lemma~\ref{lem:the-ideal-gen-QND-is-Phi-N}, the ideal in $\zz[T]$
 generated by $Q_{N,N/p}(T)$ for all $p\mid N$ is $(\Phi_N(T))$. So the
 kernel of (\ref{eq:17}) coincides with $\ker\gamma_N$. On the
 other hand, $\ker (\jmath^\vee\circ \lambda_N)$ is contained in the
 kernel of (\ref{eq:17}). Comparing dimensions, we obtain that
 \begin{equation}
   \label{eq:def-JN-new}
J_N^\nw = \ker (\jmath^\vee\circ \lambda_N)=\ker
 \gamma_N=\ker\Phi_N(\delta_N).   
 \end{equation}
 As a side result, $\pi^\vee$ must be a closed immersion since otherwise the
 kernel of $\jmath^\vee\circ \lambda_N$ will be properly contained in
 that of (\ref{eq:17}).  There is an exact sequence of abelian
 varieties
\begin{equation}
  \label{eq:18}
  0\leftarrow (J_N^{\mathrm{old}})^\vee \xleftarrow{\jmath^\vee\circ \lambda_N} J_N
  \xleftarrow{\epsilon_N} J_N^\nw \leftarrow 0.
\end{equation}
For $1<D<N$ and $D\mid N$, we have $\Phi_D(T)\mid P_D(T)$, so $\ker
\gamma_D\subseteq \ker(P_D(\delta_N))=\eta_D^*J_D$. Recall that we
have $\delta_N\eta_D^*=\eta_D^*\delta_D$ by
Subsection~\ref{subsec:basic-properties-maps}. Therefore, 
\[ \ker \gamma_D= \ker
\Phi_D(\delta_N)=\eta_D^*(\ker(J_D\xrightarrow{\Phi_D(\delta_D)}
J_D))= \eta_D^* (J_D^\nw).\]
\end{sect}

\begin{sect}
  Since $\eta_D^*:J_D\to J_N$ is a closed immersion for each $D\mid N$ and
  $D>1$, we may and will regard $J_D$ and $J_D^\nw$ as subvarieties of
  $J_N$ via $\eta_D^*$.  Let $\omega(N)$ be the number of distinct
  prime factors of $N$. If $\omega(N)=1$, then $q:=N=p^r$ is a prime
  power. In this case $J_q^{\mathrm{old}}=J_{q/p}$.  By
  \cite[Subsection 2.11]{xue-yu}, $J_q^\nw\cap J_{q/p}=
  J_{q/p}[p]$. It follows that 
  \[\forall 1\leq i \leq r-1, \qquad J_q^\nw\cap J_{p^i}=J_{p^i}[p],
  \qquad J_q^\nw \cap J_{p^i}^\nw=J_{p^i}^\nw[p]. \] 

  Now assume that $\omega(N)\geq 2$.  Let $\OO:=\zz[\zeta_N]$ be the
  ring of integers in the cyclotomic field $\qq(\zeta_N)$. By
  \cite[Proposition 2.8]{MR1421575}, $1-\zeta_N$ is a unit in
  $\OO$. For each $D\mid N$ and $1<D<N$, there is a natural action of
  $\zz[T]/(\Phi_N(T), P_D(T))$ on $J_N^\nw\cap J_D$. We have
  \[ \zz[T]/(\Phi_N(T), P_D(T))\cong \OO/ (P_D(\zeta_N))=
  \OO/((\zeta_N)^D-1)\cong \zz[T]/(\Phi_N(T), T^D-1)\] since
  $P_D(\zeta_N)=((\zeta_N)^D-1)/(\zeta_N-1)$. Clearly, $(\zeta_N)^D$
  is a primitive $(N/D)$-th root of unity, so if $\omega(N/D)>1$, then
  $(\zeta_N)^D-1$ is again a unit in $\OO$. It follows that
  \[J_N^\nw \cap J_D =\{0\}=J_N^\nw\cap J_D^\nw \qquad \text{if}\quad
  \omega(N/D)>1.\] 

  Suppose that $N=Mp^r=Dp^t$ with $r\geq t>0$ and $\gcd(p,M)=1$. In
  particular, $p\nmid \Char(K)$. Then $p\OO$ is divisible by
  $(1-(\zeta_N)^D)$, so $J_N^\nw\cap J_D\subseteq J_D[p]$.  It follows
  that $J_N^\nw\cap J_D$ is naturally a $\ff_p[T]/(\bar{\Phi}_N(T),
  T^D-1)$-module. By Lemma~\ref{lem:factorization-Phi-N-mod-p},
  $\bar{\Phi}_N(T)=\bar{\Phi}_M(T)^{\varphi(p^r)}$. On the other hand,
  $T^D-1= (T^M-1)^{p^{r-t}}$ in $\ff_p[T]$. Since $\gcd(p,M)=1$,
  $T^M-1$ is separable over $\ff_p$. Because
  $\varphi(p^r)=p^{r-1}(p-1)\geq p^{r-t}$,
  $\gcd(\bar{\Phi}_M(T)^{\varphi(p^r)}, (T^M-1)^{p^{r-t}})=
  \bar{\Phi}_M(T)^{p^{r-t}}$. We have
\[\zz[T]/(\Phi_N(T), P_D(T))\cong \ff_p[T]/(\bar{\Phi}_N(T),
T^D-1)\cong \ff_p[T]/(\bar{\Phi}_M(T)^{p^{r-t}}). \]
Therefore, if  $N=p^tD=p^rM$ with $r\geq t>0$ and $p\nmid M$, then
\begin{equation}
  \label{eq:intersection-new-part-JD}
  J_N^\nw \cap J_D=\{x\in J_D[p]\mid \Phi_M(\delta_N)^{p^{r-t}}x=0\}.
\end{equation}
By Theorem~\ref{thm:main} and
Lemma~\ref{lem:free-tate-module-p-torsion}, $J_D[p]$ is a free
$\ff_p[T]/(\bar{P}_D(T))$-module of rank $n-1$. It follows that
$J_N^\nw \cap J_D$ is a free
$\ff_p[T]/(\bar{\Phi}_M(T)^{p^{r-t}})$-module of rank $n-1$. In
particular, it has dimension $(n-1)\varphi(M)p^{r-t}$ over
$\ff_p$. Furthermore, 
\[J_N^\nw\cap J_D^\nw \subseteq J_N^\nw \cap J_D\cap J_D^\nw\subseteq
J_D[p]\cap J_D^\nw=J_D^\nw[p].\] Since $\bar{\Phi}_D(T)=
\bar{\Phi}_M(T)^{\varphi(p^{r-t})}$ divides $\bar{\Phi}_N(T)$ in
$\ff_p[T]$, one has \[J_D[p]= (\ker\Phi_D(\delta_n))[p]\subseteq (\ker
\Phi_N(\delta_N))[p]= J_N^\nw[p].\] Therefore, 
\begin{equation}
  \label{eq:JN-new-intersection-JD-nw}
  J_N^\nw\cap J_D^\nw =J_D^\nw[p]. 
\end{equation}

If $1<D_1<D_2$ are divisors of $N$ with $D_1\nmid
D_2$, then $J_{D_1}^\nw\cap J_{D_2}^\nw=\{0\}$ since $(\Phi_{D_1}(T),
\Phi_{D_2}(T))=\zz[T]$ by
Lemma~\ref{lem:cyclotomic-poly-gen-unit-ideal}.
\end{sect}

\begin{sect}\label{subsec:isomorphic-with-dual-for-JN-new}
  Recall that $G(T)=P_N(T)/\Phi_N(T)=\prod_{D\mid N,
    1<D<N}\Phi_D(T)\in \zz[T]$. Consider the map $G(\delta_N):J_N\to
  J_N$. By Subsection~\ref{subsec:def-old-and-new-part}, $\ker
  G(\delta_N)=J_N^{\mathrm{old}}$. By the proof of
  Corollary~\ref{cor:criterion-for-kernel-connected},
  $G(\delta_N)J_N=(\ker\Phi_N(\delta_N))^\circ
  =\ker\Phi_N(\delta_N)=J_N^\nw$. Therefore, $G(\delta_N)$ factors as
  $J_N\to J_N^\nw \xrightarrow{\epsilon_N} J_N$. By an abuse of
  notation, we will denote the map $J_N\to J_N^\nw$ thus obtained by
  $G(\delta_N)$ as well. There is an exact sequence
  \begin{equation}
    \label{eq:exact-seq-with-ker-JN-old}
    0\leftarrow J_N^\nw \xleftarrow{G(\delta_N)} J_N
    \xleftarrow{\jmath} J_N^{\mathrm{old}} \leftarrow 0. 
  \end{equation}
  On the other hand, taking the dual exact sequence (\cite[Exercise
  10.1, p131]{MR1987784}, or \cite[Proposition 2.42]{MR2062673} in the
  complex abelian variety case) of (\ref{eq:18}) and identifying $J_N$
  with $J_N^\vee$ via $\lambda_N$, we obtain another exact sequence 
\begin{equation}
  \label{eq:20}
  0 \leftarrow (J_N^\nw)^\vee \xleftarrow{\epsilon_N^\vee\circ \lambda_N} J_N
  \xleftarrow{\jmath} J_N^{\mathrm{old}}\leftarrow 0. 
\end{equation}
So there exists an isomorphism by comparing
(\ref{eq:exact-seq-with-ker-JN-old}) and (\ref{eq:20}):
\begin{equation}
  \label{eq:21}
  \kappa_N: J_N^\nw \xrightarrow{\simeq} (J_N^\nw)^\vee.
\end{equation}
Moreover, $G(\delta_N)=\kappa_N^{-1}\circ \epsilon_n^\vee \circ
\lambda_N$. 

The induced polarization on $J_N^\nw$ from $\lambda_N: J_N\to
J_N^\vee$ is defined to be the composition of maps
\[\lambda_N^\nw: J_N^\nw\xrightarrow{\epsilon_N} J_N
\xrightarrow{\lambda_N} J_N^\vee \xrightarrow{\epsilon_N^\vee} (J_N^\nw)^\vee. \]
It follows that $\kappa_N^{-1}\circ\lambda_N^\nw=
G(\delta_N)\mid_{J_N^\nw}$, the restriction of $G(\delta_N)$ on
$J_N^\nw$. Since $\kappa_N$ is an isomorphism, we have 
\begin{equation}
  \label{eq:22}
  \ker \lambda_N^\nw = \ker (G(\delta_N)\mid_{J_N^\nw}).
\end{equation}
\end{sect}

\begin{sect}
  Again let $\OO:=\zz[\zeta_N]\cong \zz[T]/(\Phi_N(T))$ be the ring of
  integers in the cyclotomic field $\qq(\zeta_N)$. There is an
  embedding $\OO\hookrightarrow \End(J_N^\nw)$ by $\zeta_N\mapsto
  \delta_N\mid_{J_N^\nw}$, and we will identify $\OO$ with its
  image. Then
\[G(\delta_N)\mid_{J_N^\nw}=G(\zeta_N)=\prod_{\substack{0< i<N\\\gcd(i,N)>1}}(\zeta_N-\zeta_N^i)=(\zeta_N)^{N-1-\varphi(N)}
\prod_{\substack{0<i< N\\\gcd(i,N)>1}}(1-\zeta_N^{i-1}). \]
If $\omega(N)=1$, i.e., $q:=N=p^r$ is a prime power, then
$G(T)=P_{q/p}(T)$, and $\zeta_q^{i-1}$ is a primitive $q$-th root of
unity  for all $i$ with $\gcd(i, q)>1$. Therefore, 
\begin{equation}
  \label{eq:23}
  G(\zeta_q)=u_q\cdot(1-\zeta_q)^{p^{r-1}-1} \quad \text{for some} \quad u_q\in \OO^\times.   
\end{equation}
Now suppose that $N=p_1^{r_1}\cdots p_t^{r_t}$ with $t>1$. We would
like to find the set
\[ \{ i \in \zmod{N}\mid i\neq 0, \gcd(i, N)>1, 1-\zeta_N^{i-1}\not\in
\OO^\times\}.\] First note that $1-\zeta_N^{i-1}\not\in \OO^\times$ if
and only if $\zeta_N^{(i-1)p_s^{r_s}}=1$ for some $1\leq s\leq
t$. Since it is also required that $\gcd(i, N)>1$, necessarily
$p_s\mid i$. We are reduced to solve the following equation:
\begin{equation}
  \label{eq:24}
\left\{
  \begin{aligned}
    i &\equiv 1 \mod{N/p_s^{r_s}},\\
    i &\equiv 0 \mod{p_s}.
  \end{aligned}
\right.
\end{equation}
By the Chinese Remainder Theorem, (\ref{eq:24}) has a unique solution
in $\zmod{(N/p^{r_s-1})}$. Lifting it to $\zmod{N}$, we obtain
  $p^{r_s-1}$ solutions of (\ref{eq:24}) in $\zmod{N}$. Let
  $q_s:=p_s^{r_s}$, and $\zeta_{q_s}:=(\zeta_N)^{N/q_s}$. Then
  $\zeta_N^{i-1}$ is a primitive $q_s$-th root of unity for all
  solutions of (\ref{eq:24}). It follows that if $N=\prod_{s=1}^t q_s
  =\prod_{s=1}^t p_s^{r_s}$ with $t>1$, then 
  \begin{equation}
    \label{eq:25}
    G(\zeta_N)=u_N \prod_{s=1}^t (1-\zeta_{q_s})^{p_s^{(r_s-1)}} \quad
    \text{for some} \quad u_N\in \OO^\times .
  \end{equation}
In summary, let
\begin{equation}
  \label{eq:27}
  c_N:=
  \begin{cases}
    (1-\zeta_q)^{p^{r-1}-1} &\qquad \text{if } N=q=p^r,\\
  \prod_{s=1}^t (1-\zeta_{q_s})^{p_s^{(r_s-1)}} &\qquad \text{if } N=\prod_{s=1}^t q_s
  =\prod_{s=1}^t p_s^{r_s} \text{ and } t>1.
  \end{cases}
\end{equation}
Then
\begin{equation}
  \label{eq:26}
  \ker \lambda_N^\nw
  =J_N^\nw[c_N]:=\ker(J_N^\nw\xrightarrow{c_N}J_N^\nw). 
\end{equation}
We leave it as an exercise to show that 
\begin{equation}
  \label{eq:2}
  \Nm_{\OO/\zz}(c_N)=
  \begin{cases}
   p^{\,(p^{r-1}-1)} &\qquad \text{if } N=q=p^r,\\
  \left(\prod_{p\mid N} p^{\frac{1}{p-1}}\right)^{\varphi(N)} &\qquad \text{if } \omega(N)>1.
  \end{cases}
\end{equation}
Since $T_\ell J_N^\nw$ is a free $\OO\otimes \zz_\ell$-module of rank
$n-1$ for all $\ell\neq \Char(K)$, 
\begin{equation}
  \label{eq:3}
\deg \lambda_N^\nw = \abs{\ker[\lambda_N^\nw]}=   \Nm_{\OO/\zz}(c_N)^{n-1}.  
\end{equation}
\end{sect}
\begin{thm}\label{thm:principal-polarization-Jq-nw-power-odd-prime}
  Suppose that $q=p^r$ is a prime power with $p\neq 2$.  There exists a
  principal polarization $\widetilde{\lambda}_q^\nw: J_q^\nw\to
  (J_q^\nw)^\vee$.
\end{thm}
\begin{proof}
  If $q=p$, then $J_p^\nw = J_p$ is the Jacobian of $C_{f,p}$, which
  is automatically principally polarized. So assume $r\geq 2$.  By
  Subsection~\ref{subsec:canonical-principal-polarization}, the Rosati
  involution on $\qq(\zeta_q)\subseteq \End^0(J_N^\nw)$ induced by the
  polarization $\lambda_N^\nw: J_N^\nw \to (J_N^\nw)^\vee$ is the
  complex conjugation map $x\mapsto \bar{x}$.  Since $p$ is odd,
  $(p^{r-1}-1)/2\in \nn$. We have $c_q=\tau_q^2$ with
  $\tau_q:=(1-\zeta_q)^{(p^{r-1}-1)/2}\in \zz[\zeta_q]$. Because
  $\bar{\tau}_q$ differs from $\tau_q$ by a unit,
\[\ker(J_q^\nw\xrightarrow{\lambda_q^\nw}
(J_q^\nw)^\vee)=J_q^\nw[c_q]=\ker(J_q^\nw\xrightarrow{\tau_q\bar{\tau}_q}
J_q^\nw).\] It follows that $\lambda_q^\nw:J_N^\nw \to (J_N^\nw)^\vee
$ factors as
\[ \lambda_q^\nw= \widetilde{\lambda}_q^\nw\circ
(\tau_q\bar{\tau}_q) \] for some isomorphism
$\widetilde{\lambda}_q^\nw:J_q^\nw\xrightarrow{\simeq}
(J_q^\nw)^\vee$.  Under the isomorphism
\[\Hom^0(J_q^\nw, (J_q^\nw)^\vee) \xrightarrow{\simeq} \End^0(J_q^\nw), \qquad
\phi\mapsto (\lambda_q^\nw)^{-1}\circ \phi, \]
$\widetilde{\lambda}_q^\nw$ is identified with
$1/(\tau_q\bar{\tau}_q)$. Because $1/(\tau_q\bar{\tau}_q)$ is fixed by
the Rosati involution and totally positive,
$\widetilde{\lambda}_q^\nw$ is induced from an ample line bundle $\LZ$
on $J_q^\nw$ by \cite[Application 21.III]{MR2514037}. In other words,
$\widetilde{\lambda}_q^\nw$ is a polarization, which is necessarily
principal since $\widetilde{\lambda}_q^\nw$ is an isomorphism.
\end{proof}

\begin{proof}[Proof of Theorem~\ref{thm:results-on-JN-new}] The first
  part of the theorem is proved by combining Theorem~\ref{thm:main}
  and Theorem~\ref{thm:equivalence-tate-free-kernel-connected}, noting
  that $\zz[T]/(\Phi_D(T))\cong \zz[\zeta_D]$ is integrally closed for
  any $D\in \nn$. The fact that $J_N^\nw$ is isomorphic to
  $(J_N^\nw)^\vee$ is proven in
  Subsection~\ref{subsec:isomorphic-with-dual-for-JN-new}.  The existence of a principal polarization
  $\widetilde{\lambda}_q^\nw: J_q^\nw\to (J_q^\nw)^\vee$ when $q=p^r$
  and $p\neq 2$ is shown in
  Theorem~\ref{thm:principal-polarization-Jq-nw-power-odd-prime}.
  \end{proof}

% So we have equality of ideals \[(\tau_q\bar{\tau}_q)= (\tau_q^2)=
% ((1-\zeta_q)^{q/p-1})=\p^{q/p-1}\subset \zz[\zeta_q]. \] Therefore,
% $\ker(\tau_q\bar{\tau}_q)= J_q^\nw[\p^{q/p-1}]=\ker \lambda_q^\nw$.
% It follows that there exists an isomorphism
% $\widetilde{\lambda}_q^\nw: J_q^\nw \to (J_q^\nw)^\vee$ such that
% $\lambda_q^\nw= \widetilde{\lambda}_q^\nw\circ
% (\tau_q\bar{\tau}_q)$. We have a commutative diagram
% \[
% \begin{CD}
%   (J_q^\nw)^\vee @<{\lambda_q^\nw}<< J_q^\nw @>{\lambda_q^\nw}>>
%   (J_q^\nw)^\vee \\
%    @V{\tau_q^\vee}VV @V{\bar{\tau}_q}VV
%    @AA{\widetilde{\lambda}_q^\nw}A\\
%   (J_q^\nw)^\vee @<{\lambda_q^\nw}<< J_q^\nw @>{\tau_q}>>
%   (J_q^\nw)^\vee \\
% \end{CD}\]
% The commutativity of the left square follows from the definition of
% Rosati involution with respect to $\lambda_q^\nw$ and the right from the
% definition of $\widetilde{\lambda}_q^\nw$. Hence
% \[\lambda_q^\nw \circ \bar{\tau}_q= \tau_q^\vee\circ \lambda_q^\nw = \tau_q^\vee \circ
% \widetilde{\lambda}_q^\nw \circ \tau_q \bar{\tau}_q.\]
% Canceling $\bar{\tau}_q$ from both sides, we get
% \begin{equation}
%   \label{eq:rest-polar-eq-2-pull-back-from-pp}
%   \lambda_q^\nw=\tau_q^\vee \circ
% \widetilde{\lambda}_q^\nw \circ \tau_q. 
% \end{equation}

\section{Arithmetic Results}
In this section, we prove the arithmetic results that are referred to
previously.

\begin{sect}\label{subsec:arithmetic-resultant}
  We retain the notations of
  Subsection~\ref{subsec:resultant-and-generator-of-intersection}. In
  particular, $F(T), G(T)\in \zz[T]$ are monic polynomials with
  $\gcd(F(T), G(T))=1$. Since $F(T)$ and $G(T)$ are coprime,
  $\zz[T]/(F(T), G(T))$ is a finite ring, and its cardinality is given
  by the absolute value of $\Res(F(T), G(T))$ by (\ref{eq:19}). It follows that
  \[ \Res(F(T), G(T))\in (F(T), G(T))\cap \zz,\] where the
  intersection is taken within $\zz[T]$.  Suppose that the following
  is the invariant factor decomposition \cite[Theorem
  VII.4.2]{MR1080964} of $\zz[T]/(F(T), G(T))$ as a $\zz$-module:
  \begin{equation}
    \label{eq:elementary-divisor-decomposition}
    \zz[T]/(F(T),G(T))\simeq \zmod{d_1}\oplus \cdots \oplus\zmod{d_r},
  \end{equation}
  where $d_i\mid d_{i+1}$ for all $1\leq i\leq r-1$ and $d_i>0$ for
  all $i$. We claim that
  \[ (F(T), G(T))\cap \zz=d_r\zz.\] Indeed, $(F(T), G(T))\,\cap\, \zz$ may
  be characterized as the annihilator of $\zz[T]/(F(T), G(T))$ as a
  $\zz$-module.  This coincides with $d_r\zz$ by
  (\ref{eq:elementary-divisor-decomposition}). Clearly,
  \[ \abs{\Res(F(T), G(T))}=\abs{\zz[T]/(F(T),G(T))}= d_1\cdots d_r.\]
  It follows that
  \begin{equation}
    \label{eq:result-generator-prime-divisors}
  d_r\mid \Res(F(T), G(T)),\quad \text{and} \quad p\mid
  d_r\Leftrightarrow p \mid \Res(F(T), G(T)).    
  \end{equation}
\end{sect}

Recall that $\Phi_M(T)\in\zz[T]$ denotes the $M$-th cyclotomic
polynomial, and $\bar{\Phi}_M(T)\in \ff_p[T]$ is its reduction modulo
$p$. The number of distinct prime factors of $M$ is denoted by
$\omega(M)$.
\begin{lem}\label{lem:factorization-Phi-N-mod-p}
  Suppose that $\gcd(q,D)=1$ and $q=p^r$, then $ \bar{\Phi}_{qD}(T)=
  (\bar{\Phi}_D(T))^{\varphi(q)}$ in $\ff_p[T]$. 
\end{lem}
\begin{proof}
  Let $\mu: \nn\to \{0,\pm 1\}$ be the M\"obius $\mu$-function
  (\cite[Section 2.2]{MR1070716}). More explicitly, $\mu(1)=1$, and
  for each $m>1$, $\mu(m)=(-1)^{\omega(m)}$ if $m$ is square free, and
  $\mu(m)=0$ otherwise. We have
  \[\Phi_N(T)= \prod_{D\mid N, D>0}(T^D-1)^{\mu(N/D)}. \]

% \[\Phi_N(T)= \prod_{D\mid N}(T^D-1)^{\mu(N/D)} = \prod_{D\mid N, D\neq
% 1} P_D(T)^{\mu(N/D)}. \]

Therefore, 
  \[
  \begin{split}
  \bar{\Phi}_{qD}(T)=&\prod_{m\mid qD} (T^m-1)^{\mu(qD/m)}=\prod_{m_1\mid
    D}\prod_{m_2\mid q} (T^{m_1m_2}-1)^{\mu(qD/(m_1m_2))}\\
         &=\prod_{m_1\mid D} \prod_{m_2\mid q}
         (T^{m_1}-1)^{m_2\mu(D/m_1)\mu(q/m_2)}\\
         &=\left(\prod_{m_1\mid D}(T^{m_1}-1)^{\mu(D/m_1)}\right)^{\varphi(q)}
  \end{split}
\]
since $\sum_{m_2\mid q} m_2\mu(q/m_2)=\varphi(q)$.
\end{proof}

\begin{lem}\label{lem:the-ideal-gen-QND-is-Phi-N}
  For each positive integer $D\mid N$, let
  $Q_{N,D}(T)=(T^N-1)/(T^D-1)\in\zz[T]$ be the polynomial in
  {\rm{(\ref{eq:10})}}.  The ideal in $\zz[T]$ generated by $Q_{N,
    N/p}(T)$ for all primes $p\mid N$ is the principal ideal
  $(\Phi_N(T))$.
\end{lem}
\begin{proof}
  If $N=p^r$ is a prime power, then $Q_{N,N/p}(T)=\Phi_{p^r}(T)$.  Now
  suppose that $N=\prod_{i=1}^t p_i^{r_i}$ with $t>1$. For each $p\mid
  N$, let \[F_p(T):=\frac{T^N-1}{(T^{(N/p)}-1)\Phi_N(T)}\in
  \zz[T].\] Since $\zz[T]$ is a unique factorization domain, it is
  enough to show that the ideal $I$ generated by all $F_p(T)$ with
  $p\mid N$ is the unit ideal in $\zz[T]$. Suppose otherwise, then $I$
  is contained in a maximal ideal $\m$ of $\zz[T]$. By
  Subsection~\ref{subsec:arithmetic-resultant} and (\ref{eq:14}),
  $\m\cap \zz=p\zz$ for some $p\mid N$. Let $\bar{I}$ be the canonical
  image of $I$ in the quotient ring $\ff_p[T]=\zz[T]/(p)$. It is
  contained in the maximal ideal $\bar{\m}\subset \ff_p[T]$. Suppose
  that $N=qD$ with $q=p^r$ and $p\nmid D$.  Let $\ell$ be a prime
  divisor of $D$.
  \begin{align*}
   \bar{F}_p(T)= \frac{T^N-1}{(T^{(N/p)}-1)\bar{\Phi}_N(T)}&=\left(\frac{T^D-1}{\bar{\Phi}_D(T)}\right)^{\varphi(q)}=
    \left(\prod_{M\mid D, \; M\neq D}
      \bar{\Phi}_M(T)\right)^{\varphi(q)},
    \\
\bar{F}_\ell(T)=    \frac{T^N-1}{(T^{(N/\ell)}-1)\bar{\Phi}_N(T)}&=\frac{
      (T^D-1)^q}{(T^{(D/\ell)}-1)^q
      \bar{\Phi}_D(T)^{\varphi(q)}}=\bar{\Phi}_D(T)^{q/p}\prod_{d\mid
      D,\; \ell d\nmid D,}^{d\neq D} \bar{\Phi}_d(T)^q.
  \end{align*}
  Because $\gcd(p,D)=1$, the polynomial $T^D-1=\prod_{M\mid
    D}\bar{\Phi}_M(T)$ is separable in $\ff_p[T]$. In particular, for
  any two distinct divisors $M_1, M_1$ of $D$,
  $\gcd(\bar{\Phi}_{M_1}(T), \bar{\Phi}_{ M_2}(T))=1$.  Clearly,
  $\bar{F}_p(T)$ is not divisible by $\bar{\Phi}_D(T)$. For any $M\mid
  D$ and $M\neq D$, we take $\ell$ to be a prime divisor of $D/M$,
  then $\gcd(\bar{F}_\ell(T),\bar{\Phi}_M(T))=1$. It follows that
  $\gcd(\bar{F}_{p_1}(T), \cdots,\bar{F}_{p_t}(T))=1$ and hence
  $\bar{I}=\ff_p[T]$, which leads to a contradiction.
\end{proof}

\begin{lem}\label{lem:cyclotomic-poly-gen-unit-ideal}
  Suppose $D_1<D_2$ are two positive integers with $D_1\nmid
  D_2$. Then
  \[(\Phi_{D_1}(T), \Phi_{D_2}(T))=\zz[T].\]
\end{lem}
\begin{proof}
  Let $d=\gcd(D_1, D_2)<D_1$.  Clearly, $(\Phi_{D_1}(T),
  \Phi_{D_2}(T))\supseteq (Q_{D_1,d}(T), Q_{D_2, d}(T))$. So it is
  enough to prove that $(Q_{D_1,d}(T), Q_{D_2, d}(T))=\zz[T]$.  By the
  Euclidean algorithm, there exist $a(T), b(T)\in \zz[T]$ such that
\[ a(T)(T^{D_1}-1)+b(T)(T^{D_2}-1)=T^d-1.\]
The lemma follows by diving both sides by $T^d-1$. 
\end{proof}

%%%%%%%%%%%%%%%%%%%%%% AuCTeX  Tips %%%%%%%%%%%%
% (C-c %) to remove comments

%%%%%%%%%%%%%%%%%%%%%%%%%%%%%%%%%%%%%%%%%%%%%%%%
\bibliographystyle{plain}
\bibliography{TeXBiB}

\def\cprime{$'$}
\begin{thebibliography}{10}

\bibitem{MR2062673}
Christina Birkenhake and Herbert Lange.
\newblock {\em Complex abelian varieties}, volume 302 of {\em Grundlehren der
  Mathematischen Wissenschaften [Fundamental Principles of Mathematical
  Sciences]}.
\newblock Springer-Verlag, Berlin, second edition, 2004.

\bibitem{MR1080964}
N.~Bourbaki.
\newblock {\em Algebra. {II}. {C}hapters 4--7}.
\newblock Elements of Mathematics (Berlin). Springer-Verlag, Berlin, 1990.
\newblock Translated from the French by P. M. Cohn and J. Howie.

\bibitem{MR1322960}
David Eisenbud.
\newblock {\em Commutative algebra with a view toward algebraic geometry},
  volume 150 of {\em Graduate Texts in Mathematics}.
\newblock Springer-Verlag, New York, 1995.

\bibitem{MR1708603}
Gerhard Frey, Ernst Kani, and Helmut V{\"o}lklein.
\newblock Curves with infinite {$K$}-rational geometric fundamental group.
\newblock In {\em Aspects of {G}alois theory ({G}ainesville, {FL}, 1996)},
  volume 256 of {\em London Math. Soc. Lecture Note Ser.}, pages 85--118.
  Cambridge Univ. Press, Cambridge, 1999.

\bibitem{MR0463157}
Robin Hartshorne.
\newblock {\em Algebraic geometry}.
\newblock Springer-Verlag, New York, 1977.
\newblock Graduate Texts in Mathematics, No. 52.

\bibitem{MR1070716}
Kenneth Ireland and Michael Rosen.
\newblock {\em A classical introduction to modern number theory}, volume~84 of
  {\em Graduate Texts in Mathematics}.
\newblock Springer-Verlag, New York, second edition, 1990.

\bibitem{MR1107394}
Ja~Kyung Koo.
\newblock On holomorphic differentials of some algebraic function field of one
  variable over {${\bf C}$}.
\newblock {\em Bull. Austral. Math. Soc.}, 43(3):399--405, 1991.

\bibitem{MR861976}
J.~S. Milne.
\newblock Jacobian varieties.
\newblock In {\em Arithmetic geometry ({S}torrs, {C}onn., 1984)}, pages
  167--212. Springer, New York, 1986.

\bibitem{MR2987306}
Shinichi Mochizuki.
\newblock Topics in absolute anabelian geometry {I}: generalities.
\newblock {\em J. Math. Sci. Univ. Tokyo}, 19(2):139--242, 2012.

\bibitem{MR2514037}
David Mumford.
\newblock {\em Abelian varieties}, volume~5 of {\em Tata Institute of
  Fundamental Research Studies in Mathematics}.
\newblock Published for the Tata Institute of Fundamental Research, Bombay,
  2008.
\newblock With appendices by C. P. Ramanujam and Yuri Manin, Corrected reprint
  of the second (1974) edition.

\bibitem{MR1987784}
Alexander Polishchuk.
\newblock {\em Abelian varieties, theta functions and the {F}ourier transform},
  volume 153 of {\em Cambridge Tracts in Mathematics}.
\newblock Cambridge University Press, Cambridge, 2003.

\bibitem{MR1465369}
Bjorn Poonen and Edward~F. Schaefer.
\newblock Explicit descent for {J}acobians of cyclic covers of the projective
  line.
\newblock {\em J. Reine Angew. Math.}, 488:141--188, 1997.

\bibitem{MR1612262}
Edward~F. Schaefer.
\newblock Computing a {S}elmer group of a {J}acobian using functions on the
  curve.
\newblock {\em Math. Ann.}, 310(3):447--471, 1998.

\bibitem{Serre_local}
Jean-Pierre Serre.
\newblock {\em Local fields}, volume~67 of {\em Graduate Texts in Mathematics}.
\newblock Springer-Verlag, New York, 1979.
\newblock Translated from the French by Marvin Jay Greenberg.

\bibitem{MR1421575}
Lawrence~C. Washington.
\newblock {\em Introduction to cyclotomic fields}, volume~83 of {\em Graduate
  Texts in Mathematics}.
\newblock Springer-Verlag, New York, second edition, 1997.

\bibitem{MR2831828}
Jiangwei Xue.
\newblock Hodge groups of certain superelliptic {J}acobians {II}.
\newblock {\em Math. Res. Lett.}, 18(4):579--590, 2011.

\bibitem{xue-yu}
Jiangwei Xue and Chia-Fu Yu.
\newblock Endomorphism algebras of factors of certain hypergeometric
  {J}acobians.
\newblock {\em http://arxiv.org/abs/1304.6202}, 2013.

\bibitem{MR2657449}
Jiangwei Xue and Yuri~G. Zarhin.
\newblock Centers of {H}odge groups of superelliptic {J}acobians.
\newblock {\em Transform. Groups}, 15(2):449--482, 2010.

\bibitem{MR2644383}
Jiangwei Xue and Yuri~G. Zarhin.
\newblock Hodge groups of certain superelliptic {J}acobians.
\newblock {\em Math. Res. Lett.}, 17(2):371--388, 2010.

\bibitem{MR2040573}
Yuri~G. Zarhin.
\newblock The endomorphism rings of {J}acobians of cyclic covers of the
  projective line.
\newblock {\em Math. Proc. Cambridge Philos. Soc.}, 136(2):257--267, 2004.

\bibitem{MR2166091}
Yuri~G. Zarhin.
\newblock Endomorphism algebras of superelliptic {J}acobians.
\newblock In {\em Geometric methods in algebra and number theory}, volume 235
  of {\em Progr. Math.}, pages 339--362. Birkh\"auser Boston, Boston, MA, 2005.

\bibitem{MR2349666}
Yuri~G. Zarhin.
\newblock Superelliptic {J}acobians.
\newblock In {\em Diophantine geometry}, volume~4 of {\em CRM Series}, pages
  363--390. Ed. Norm., Pisa, 2007.

\bibitem{MR2471095}
Yuri~G. Zarhin.
\newblock Endomorphisms of superelliptic {J}acobians.
\newblock {\em Math. Z.}, 261(3):691--707, 2009.
\newblock Updated in 2012 at http://arxiv.org/pdf/math/0605028.pdf.

\end{thebibliography}
\end{document}